\documentclass{amsart}

\usepackage{amsmath,amsfonts,amssymb,amsthm,graphicx,cite}

\usepackage{tikz-cd}

\newtheorem{theorem}{Theorem}[section]
\newtheorem{corollary}[theorem]{Corollary}
\newtheorem{lemma}[theorem]{Lemma}
\newtheorem{proposition}[theorem]{Proposition}
\theoremstyle{remark}
\newtheorem{remark}[theorem]{Remark}

\newcommand{\N}{{\mathbb N}}
\newcommand{\C}{{\mathbb C}}

\newcommand{\cD}{{\mathcal D}}
\newcommand{\cH}{{\mathcal H}}
\newcommand{\cM}{{\mathcal M}}
\newcommand{\cR}{{\mathcal R}}

\begin{document}

\title[From Kajihara's formula to deformed MR and NS operators]{From Kajihara's transformation formula to deformed Macdonald--Ruijsenaars and Noumi--Sano operators}

\author[M.~Halln\"as]{Martin Halln\"as}
\address{Department of Mathematical Sciences, Chalmers University of Technology and the University of Gothenburg, SE-412 96 Gothenburg, Sweden}
\email{hallnas@chalmers.se}

\author[E.~Langmann]{Edwin Langmann}
\address{Department of Physics, KTH Royal Institute of Technology, SE-106 91 Stockholm, Sweden}
\email{langmann@kth.se}

\author[M.~Noumi]{Masatoshi Noumi}
\address{Department of Mathematics, KTH Royal Institute of Technology, SE-100 44 Stockholm, Sweden (on leave from: Department of Mathematics, Kobe University, Rokko, Kobe 657-8501, Japan)}
\email{noumi@math.kobe-u.ac.jp}

\author[H.~Rosengren]{Hjalmar Rosengren}
\address{Department of Mathematical Sciences, Chalmers University of Technology and the University of Gothenburg, SE-412 96 Gothenburg, Sweden}
\email{hjalmar@chalmers.se}

\date{\today}

\begin{abstract}
Kajihara obtained in 2004 a remarkable transformation formula connecting multiple basic hypergeometric series  associated with $A$-type root systems of different ranks. By multiple principle specialisations of his formula, we deduce kernel identities for deformed Macdonald--Ruijsenaars (MR) and Noumi--Sano (NS) operators. The deformed MR operators were introduced by Sergeev and Veselov in the first order case and by Feigin and Silantyev in the higher order cases.
As applications of our kernel identities, we prove that all of these operators pairwise commute and are simultaneously diagonalised by the super-Macdonald polynomials. We also provide an explicit description of the algebra generated by the deformed MR and/or NS operators by a Harish-Chandra type isomorphism and show that the deformed MR (NS) operators can be viewed as restrictions of inverse limits of ordinary MR (NS) operators.
\end{abstract}

\maketitle

\tableofcontents

\section{Introduction}
In the mid 1840s, Heine \cite{Hei46,Hei47} introduced the basic hypergeometric series
\begin{equation}
\label{2phi1}
{}_2 \phi_1\left[
\setlength\arraycolsep{1pt}
\begin{array}{c}
a, b\\
c
\end{array}
;q,u\right]
= \sum_{k=0}^\infty \frac{(a;q)_k(b;q)_k}{(q;q)_k(c;q)_k}u^k,
\end{equation}
with the $q$-Pochhammer symbol
$$
(a;q)_k = \frac{(a;q)_\infty}{(aq^k;q)_\infty} = (1-a)(1-aq)\cdots (1-aq^{k-1}),\ \ \ (a;q)_\infty = \prod_{n=0}^\infty (1-aq^n),
$$
as a natural $q$-deformation of Gauss' hypergeometric series ${}_2 F_1(a,b;c;z)$. For a detailed account of such series, see e.g.~Gasper and Rahman's book \cite{GS04}. Among Heine's many fundamental results is the transformation formula  
\begin{equation}
\label{Heine}
{}_2 \phi_1\left[
\setlength\arraycolsep{1pt}
\begin{array}{c}
a, b\\
c
\end{array}
;q,u\right]
= \frac{(abu/c;q)_\infty}{(u;q)_\infty}
{}_2 \phi_1\left[
\setlength\arraycolsep{1pt}
\begin{array}{c}
c/a, c/b\\
c
\end{array}
;q,abu/c\right],
\end{equation}
which can be viewed as a $q$-analogue of Euler's transformation formula for ${}_2 F_1$.

Kajihara's formula \cite{Kaj04} (see also \cite{KN03}), which is our starting point in this paper, is a far-reaching generalisation of Heine's formula \eqref{Heine}, connecting multiple basic hypergeometric series associated with root systems of type A of different ranks.

Other key objects in the paper are particular generalisations of the Macdonald--Ruijsenaars (MR) $q$-difference operators $D_n^r$, $r=0,1,\ldots,n$. From Chapter VI in Macdonald's book \cite{Mac95}, we recall the elegant and explicit definition in terms of the generating series
\begin{equation}
\label{Dn}
\begin{split}
D_n(x;u) &= \sum_{I\subseteq\{1,\ldots,n\}} (-u)^{|I|}t^{\binom{|I|}{2}} \prod_{i\in I, j\notin I} \frac{tx_i-x_j}{x_i-x_j}\cdot \prod_{i\in I}T_{q,x_i}\\
&= \sum_{r=0}^n (-u)^r D_n^r(x),
\end{split}
\end{equation}
where $T_{q,x_i}$ denotes the $q$-shift operator with respect to $x_i$. (Here and below, we suppress the dependence on the parameters $q$ and $t$ whenever ambiguities are unlikely to arise.)  Up to a change of gauge and variables, the $q$-difference operators $D_n^r$ coincide with the trigonometric version of the difference operators $\hat{S_r}$ introduced by Ruijsenaars' \cite{Rui87}, who proved that they commute, and thus define a quantum integrable system. We note that he obtained these results even at the more general elliptic level.

In a more recent development, Noumi and Sano \cite{NS20} introduced an infinite family of commuting $q$-difference operators $H_n^r$, $r\in\N$, given by the expansion
\begin{equation}
\label{Hn}
\begin{split}
H_n(x;u) &= \sum_{\mu\in\N^n} u^{|\mu|} \frac{\Delta(q^\mu x)}{\Delta(x)} \prod_{i,j=1}^n \frac{(tx_i/x_j;q)_{\mu_i}}{(qx_i/x_j;q)_{\mu_i}}\cdot \prod_{i=1}^n T_{q,x_i}^{\mu_i}\\
&= \sum_{r=0}^\infty u^rH_n^r(x),
\end{split}
\end{equation}
and proved that they generate the same commutative algebra as the MR operators $D_n^r$, $r=1,\ldots,n$, to which they are related through a Wronski-type formula. Throughout the paper, we refer to the operators $H_n^r$ as the Noumi--Sano (NS) operators.

Remarkably, the MR and NS operators can be unified in a family of commuting difference operators $D_{n,m}^r(x,y;q,t)$, $r\in\N$, in two sets of variables $x=(x_1,\ldots,x_n)$ and $y=(y_1,\ldots,y_m)$, which reduce to $D_n^r(x;q,t)$ and $H_m^r(y;q^{-1},t^{-1})$ for $m=0$ and $n=0$, respectively. Such difference operators first appeared in the $m=1$ case in work by Chalykh \cite{Cha00,Cha02}. Sergeev and Veselov \cite{SV04,SV09} introduced and studied the $r=1$ operators for general $n,m\in\N$, while the $r>2$ operators are due to Feigin and Silantyev \cite{FS14}, who, in particular, proved commutativity. The operators $D_{n,m}^r$ can be considered as natural difference analogues of so-called deformed (trigonometric) Calogero--Moser--Sutherland operators \cite{CFV98,Ser02,SV04,SV05}, which, in turn, are intimately related to Lie superalgebras \cite{Ser02,SV04}, $\beta$-ensembles of random matrices \cite{DL15}, as well as conformal field theory and the fractional quantum Hall effect \cite{AL17}.

In this paper, we establish an intriguing connection between Kajihara's transformation formula and Feigin and Silantyev's difference operators $D_{n,m}^r(x,y)$: By multiple principle specialisations of the former, we obtain so-called kernel identities of the form
$$
(\cD_{n,m}(x,y;u)-\cD_{N,M}(z,w;u))\Phi_{n,m;N,M}(x,y;z,w)=0
$$
for the renormalised generating series
$$
\cD_{n,m}(x,y;u) = \frac{(q^mt^{-n}u;t^{-1})_\infty}{(u;t^{-1})_\infty} \sum_{r=0}^\infty (-u)^r D_{n,m}^r(x,y),
$$
where $(n,m),(N,M)\in\N^2$ can be chosen arbitrarily. For the $r=1$ difference operators such identities were previously obtained by Atai and two of the authors \cite{AHL14}. Our kernel function $\Phi_{n,m;N,M}$ is an explicitly given meromorphic function, which reduces to Macdonald's (reproducing) kernel function $\Pi(x,z)$ when $m=M=0$; see, e.g., Section VI.3 in \cite{Mac95} for corresponding kernel identities involving $D_n$ \eqref{Dn} and \cite{NS20} for identities relating $H_n$ with $H_m$ and $D_n$ with $H_m$.

In addition, we obtain kernel  identities involving a `dual' family of difference operators $H_{n,m}^r(x,y;q,t)$, $r\in\N$, in which the roles of the two sets of variables are interchanged, and which specialise to $H_n^r(x;q,t)$ when $m=0$ and $D_m^r(y;t^{-1},q^{-1})$ in case $n=0$.

In keeping with earlier literature on the subject, we shall refer to the difference operators $D_{n,m}^r$ and $H_{n,m}^r$ as deformed MR operators and deformed NS operators, respectively. Their precise definition is given in Section \ref{Sec:kIds}, where the corresponding kernel  identities, alluded to above, are also formulated and proved. 

In Section \ref{Sec:Apps}, we detail a number of applications of our kernel identities. From the known commutativity of the ordinary MR and NS operators, we infer in Subsection \ref{SubSec:Comm} that their deformed counterparts all commute with each other. In Subsection \ref{SubSec:EigFuncs}, we show that the so-called super-Macdonald polynomials \cite{SV09} are joint eigenfunctions of the deformed MR and NS operators, and we also compute the corresponding joint eigenvalues explicitly. Again, we rely on known eigenfunction properties of the ordinary MR and NS operators. Subsection \ref{SubSec:HC} contains a simple and explicit description of the commutative algebra $\cR_{n,m}$ generated by the deformed NS operators $H_{n,m}^r$ ($r\in\N$) by an Harish-Chandra type isomorphism to an algebra of polynomials in $n+m$ variables with suitable symmetry properties. As corollaries, we establish Wronski type recurrence relations for the deformed MR and NS operators, and thereby show that the deformed MR operators $D_{n,m}^r$ ($r\in\N$) provide another set of generators for $\cR_{n,m}$. In addition, we infer that the first $n+m$ operators $D_{n,m}^r$, or alternatively  $H_{n,m}^r$, are algebraically independent, and thus define an integrable system. For the former operators, this was first shown by Feigin and Silantyev \cite{FS14}. Finally, in Subsection \ref{SubSec:Res}, we provide an interpretation of the deformed MR and NS operators as particular restrictions of operators on the algebra of (complex) symmetric functions. This generalises results of Sergeev and Veselov \cite{SV09} on the $r=1$ case.

As we prove in the paper \cite{HLNR21}, some of these results, including kernel identities and commutativity, generalise to the elliptic level.

The proofs of the various lemmas in the main text are collected in Appendix \ref{App:Lemmas}.

\subsection*{Notation}
We use the convention $\N=\{0,1,2,\ldots\}$ and let $\N^*=\N\setminus\{0\}$. Unless otherwise specified, we follow Macdonald's book \cite{Mac95} for notation and terminology from the theory of symmetric functions.

\section{Kernel identities}
\label{Sec:kIds}
This section is devoted to the formulation and proof of our main result. To this end, we recall in Subsection \ref{SubSec:Kaj} Kajihara's transformation formula, whereas Subsection \ref{SubSec:defOps} contains definitions of the deformed MR and NS operators. The corresponding kernel identities are then stated and proved in Subsection \ref{SubSec:kerIds}.

\subsection{Kajihara's transformation formula}
\label{SubSec:Kaj}
Let $K,L\in\N^*$. Given four vectors of (complex) variables
$$
(a_1,\ldots,a_K), (X_1,\ldots,X_K)\in\C^K,\ \ (b_1,\ldots,b_L), (c_1,\ldots,c_L)\in\C^L,
$$
we recall Kajihara and Noumi's \cite{KN03} multiple basic hypergeometric series
\begin{multline}
\label{phiKL}
\setlength\arraycolsep{1pt}
\phi^{K,L}\left(
\begin{array}{c}
a_1,\ldots,a_K\\
X_1,\ldots,X_K
\end{array}
\hspace{1pt}
\vline
\hspace{1pt}
\begin{array}{c}
b_1,\ldots,b_L\\
c_1,\ldots,c_L
\end{array}
;u\right)\\
= \sum_{\gamma\in\N^K} u^{|\gamma|} \frac{\Delta(q^\gamma X)}{\Delta(X)} \prod_{i,j=1}^K \frac{(a_jX_i/X_j;q)_{\gamma_i}}{(qX_i/X_j;q)_{\gamma_i}} \cdot \prod_{i=1}^K \prod_{k=1}^L \frac{(X_ib_k;q)_{\gamma_i}}{(X_ic_k;q)_{\gamma_i}},
\end{multline}
where
$$
\Delta(X) = \prod_{1\leq i<j\leq K} (X_i-X_j).
$$
For general values of $a_j$ ($j=1,\ldots,K$), $b_k$ ($k=1,\ldots,L$) and $c\in\C$, Kajihara \cite{Kaj04} established the transformation formula
\begin{multline}
\label{KajForm}
\phi^{K,L}\left(
\setlength\arraycolsep{1pt}
\begin{array}{c}
a_1,\ldots,a_K\\
X_1,\ldots,X_K
\end{array}
\hspace{1pt}
\vline
\hspace{1pt}
\begin{array}{c}
b_1Y_1,\ldots,b_LY_L\\
cY_1,\ldots,cY_L
\end{array}
;u\right)\\
= \frac{(\alpha\beta u/c^L;q)_\infty}{(u;q)_\infty} \phi^{L,K}\left(
\setlength\arraycolsep{1pt}
\begin{array}{c}
c/b_1,\ldots,c/b_L\\
Y_1,\ldots,Y_L
\end{array}
\hspace{1pt}
\vline
\hspace{1pt}
\begin{array}{c}
cX_1/a_1,\ldots,cX_K/a_K\\
cX_1,\ldots,cX_K
\end{array}
;\alpha\beta u/c^L\right),
\end{multline}
where
$$
\alpha := a_1\cdots a_K,\ \ \ \beta := b_1\cdots b_L.
$$
In the special case $K=L=1$, it is readily seen that
$$
\phi^{1,1}\left(
\setlength\arraycolsep{1pt}
\begin{array}{c}
a\\
X
\end{array}
\hspace{1pt}
\vline
\hspace{1pt}
\begin{array}{c}
b\\
c
\end{array}
;u\right)
=
{}_2 \phi_1\left[
\setlength\arraycolsep{1pt}
\begin{array}{c}
a, bX\\
cX
\end{array}
;q,u\right]
$$
and that \eqref{KajForm} reduces to Heine's transformation formula \eqref{Heine} (with $b\to bXY$ and $c\to cXY$).

\subsection{Deformed NS and MR operators}
\label{SubSec:defOps}
Here and throughout the paper, we assume that $q,t\in\C^*$ are not roots of unity to ensure that, in particular, all of the operators in question are well-defined.

To an $n$-tuple $\mu\in\N^n$, we associate the $q$-difference operator
$$
T_{q,x}^{\pm \mu} := \prod_{i=1}^n T_{q,x_i}^{\pm\mu_i},
$$
which acts on meromorphic functions in $x=(x_1,\ldots,x_n)$ according to
$$
T_{q,x}^{\pm\mu} f(x_1,\ldots,x_n) = f(q^{\pm\mu_1}x_1,\ldots,q^{\pm\mu_n}x_n).
$$
Moreover, we find it convenient to identify subsets $I\subseteq\{1,\ldots,m\}$ with $m$-tuples $(I_1,\ldots,I_m)\in\{0,1\}^m$, where $I_i=1$ when $i\in I$ and $I_i=0$ otherwise. With this identification in place, we have
$$
T_{t,y}^{\pm I} = \prod_{i\in I}T_{t,y_i}^{\pm 1}.
$$

We can now define the deformed NS operators $H_{n,m}^r$ ($r\in\N$) by the generating series
\begin{equation}
\label{Hnm}
\begin{split}H_{n,m}(x,y;u) &= \sum_{\mu\in\N^n}\sum_{I\subseteq\{1,\ldots,m\}} (t^{1-n}q^mu)^{|\mu|} (-tu)^{|I|}q^{\binom{|I|}{2}} B_{\mu,I}(x,y) T_{q,x}^{\mu}T_{t,y}^{-I}\\
&= \sum_{r=0}^\infty u^rH_{n,m}^r(x,y;q,t),
\end{split}
\end{equation}
with coefficient functions
\begin{equation}
\label{BmuI}
\begin{split}
B_{\mu,I}(x,y) &= \frac{\Delta(q^{\mu} x)}{\Delta(x)} \prod_{i,j=1}^n \frac{(tx_i/x_j;q)_{\mu_i}}{(qx_i/x_j;q)_{\mu_i}}\cdot \prod_{\substack{1\leq i,j\leq m\\ i\in I; j\notin I}} \frac{y_i-qy_j}{y_i-y_j}\\
&\quad \cdot \prod_{i=1}^n \Bigg(\prod_{j\in I}\frac{1-x_i/ty_j}{1-q^{\mu_i}x_i/y_j}\cdot \prod_{j\notin I}\frac{1-x_i/qy_j}{1-q^{\mu_i-1}x_i/y_j}\Bigg),
\end{split}
\end{equation}
and where $|\mu|=\sum_{i=1}^m \mu_i$ and $|I|$ denotes the cardinality of $I$.

Setting $m=0$ in \eqref{Hnm}--\eqref{BmuI} and comparing the resulting expressions with \eqref{Hn}, we see that $H_{n,0}(x;u;q,t)=H_n(x;t^{1-n}u;q,t)$. On the other hand, taking $n=0$, we find that $H_{0,m}(y;u;q,t)=D_m(y;tq^{m-1}u;t^{-1},q^{-1})$, cf.~\eqref{Dn}.

We obtain the deformed MR operators by interchanging $n\leftrightarrow m$, $x\leftrightarrow y$ and $q\leftrightarrow t^{-1}$ as well as scaling $u\to qu$ in the deformed NS operators. More precisely, we have
\begin{equation}
\label{Dnm}
\begin{split}
D_{n,m}(x,y;u;q,t) &= H_{m,n}(y,x;qu;t^{-1},q^{-1})\\
&= \sum_{r=0}^\infty (-u)^rD_{n,m}^r(x,y;q,t).
\end{split}
\end{equation}
Introducing the coefficient functions
\begin{equation}
\label{AImu}
\begin{split}
A_{I,\mu}(x,y) &= \prod_{\substack{1\leq i,j\leq n\\ i\in I; j\notin I}} \frac{x_i-t^{-1}x_j}{x_i-x_j}\cdot \frac{\Delta(t^{-\mu} y)}{\Delta(y)} \prod_{i,j=1}^m \frac{(y_i/qy_j;t^{-1})_{\mu_i}}{(y_i/ty_j;t^{-1})_{\mu_i}}\\
&\quad \cdot \prod_{i=1}^m \Bigg(\prod_{j\in I}\frac{1-qy_i/x_j}{1-t^{-\mu_i}y_i/x_j}\cdot \prod_{j\notin I}\frac{1-ty_i/x_j}{1-t^{1-\mu_i}y_i/x_j}\Bigg),
\end{split}
\end{equation}
we get the explicit formula
\begin{equation*}
\begin{split}
D_{n,m}(x,y;u) &= \sum_{I\subseteq\{1,\ldots,n\}}\sum_{\mu\in\N^m} (q^m t^{-n}u)^{|\mu|} (-u)^{|I|}t^{-\binom{|I|}{2}} A_{I,\mu}(x,y)T_{q,x}^IT_{t,y}^{-\mu}\\
&= \sum_{r=0}^\infty (-u)^r D_{n,m}^r(x,y),
\end{split}
\end{equation*}

We note the special cases $D_{n,0}(x;u)=D_n(x;t^{1-n}u)$ and $D_{0,m}(y;u;q,t)=H_m(y;q^{m}u;t^{-1},q^{-1})$. In the general case, $D_{n,m}^r$ should be compared with $\overline{M_{-b_r}}$ in Eq.~(4.19) of \cite{FS14}. Indeed, after invoking the elementary identity
\begin{equation*}
\frac{\Delta(t^{-\mu} y)}{\Delta(y)} \prod_{i,j=1}^n \frac{1}{(y_i/ty_j;t^{-1})_{\mu_i}}
= (-1)^{|\mu|}t^{|\mu|(|\mu|+1)/2}\prod_{i,j=1}^n \frac{1}{(t^{\mu_j}y_i/y_j;t^{-1})_{\mu_i}},
\end{equation*}
it is readily seen that the former may be viewed as a multiplicative form of the latter additive difference operators.

\subsection{Kernel identities}
\label{SubSec:kerIds}
We proceed to state and prove our kernel identities for the deformed MR and NS operators.

Our result takes a particularly simple form when expressed in terms of the modified generating series
\begin{equation}
\label{cHnm}
\cH_{n,m}(x,y;u) = \frac{(t^{1-n}q^mu;q)_\infty}{(tu;q)_\infty}H_{n,m}(x,y;u)
\end{equation}
and
\begin{equation}
\label{cDnm}
\cD_{n,m}(x,y;u) = \frac{(q^m t^{-n}u;t^{-1})_\infty}{(u;t^{-1})_\infty}D_{n,m}(x,y;u),
\end{equation}
cf.~\eqref{Dnm} and \eqref{Hnm}.

Under the assumption that $|q|<1$ and $|t|>1$, we define the meromorphic function $\Phi_{n,m;N,M}(x,y;z,w)=\Phi_{n,m;N,M}(x,y;z,w;q,t)$ in $n+m$ variables $(x,y)=((x_1,\ldots,x_n),(y_1,\ldots,y_m))$ and $N+M$ variables $(z,w)=((z_1,\ldots,z_N),(w_1,\ldots,w_M))$ by
\begin{equation}
\label{kerFunc}
\begin{split}
\Phi_{n,m;N,M}(x,y;z,w) &= \prod_{i=1}^n\prod_{j=1}^N\frac{(x_iz_j;q)_\infty}{(t^{-1}x_iz_j;q)_\infty}\cdot \prod_{i=1}^m\prod_{j=1}^M\frac{(y_iw_j;t^{-1})_\infty}{(qy_iw_j;t^{-1})_\infty}\\
&\quad \cdot \prod_{i=1}^n\prod_{j=1}^M (1-x_iw_j)\cdot \prod_{i=1}^m\prod_{j=1}^N (1-y_iz_j).
\end{split}
\end{equation}

The following theorem constitutes our main result.

\begin{theorem}
\label{Thm:kerIds}
For $0<|q|<1$ and $|t|>1$, we have the kernel identities
\begin{equation}
\label{cHkerId}
\cH_{n,m}(x,y;u)\Phi_{n,m;N,M}(x,y;z,w) = \cH_{N,M}(z,w;u)\Phi_{n,m;N,M}(x,y;z,w)
\end{equation}
and
\begin{equation}
\label{cDkerId}
\cD_{n,m}(x,y;u)\Phi_{n,m;N,M}(x,y;z,w) = \cD_{N,M}(z,w;u)\Phi_{n,m;N,M}(x,y;z,w).
\end{equation}
\end{theorem}

\begin{proof}
From \eqref{Hnm} and \eqref{cHnm}--\eqref{cDnm}, we have $\cH_{n,m}(x,y;u;q,t)=\cD_{m,n}(y,x;tu;t^{-1},q^{-1})$. Hence, thanks to the manifest symmetry property
$$
\Phi_{n,m;N,M}(x,y;z,w;q,t) = \Phi_{m,n;M,N}(y,x;w,z;t^{-1},q^{-1}),
$$
it suffices to prove the kernel identity \eqref{cHkerId}, say.

Taking $c=1$, $b_k\to 1/b_k$ (which entails $\beta\to1/\beta$) and $u\to u/\alpha$ in \eqref{KajForm} and substituting the expression \eqref{phiKL} for $\phi^{K,L}$, we deduce the identity
\begin{multline}
\label{KajForm2}
\frac{(u/\alpha;q)_\infty}{(u;q)_\infty}\sum_{\gamma\in\N^K} (u/\alpha)^{|\gamma|} \frac{\Delta(q^\gamma X)}{\Delta(X)}\\ \cdot \prod_{i,j=1}^K \frac{(a_jX_i/X_j;q)_{\gamma_i}}{(qX_i/X_j;q)_{\gamma_i}}\cdot \prod_{i=1}^K\prod_{k=1}^L \frac{(X_iY_k/b_k;q)_{\gamma_i}}{(X_iY_k;q)_{\gamma_i}}\\
= \frac{(u/\beta;q)_\infty}{(u;q)_\infty}\sum_{\delta\in\N^L} (u/\beta)^{|\delta|} \frac{\Delta(q^\delta Y)}{\Delta(Y)}\\ \cdot \prod_{k,l=1}^L \frac{(b_lY_k/Y_l;q)_{\delta_k}}{(qY_k/Y_l;q)_{\delta_k}}\cdot \prod_{k=1}^L\prod_{i=1}^K \frac{(Y_kX_i/a_i;q)_{\delta_k}}{(Y_kX_i;q)_{\delta_k}}.
\end{multline}
Choosing $K=n+m$ and $L=N+M$, we specialise the variables according to 
\begin{equation}
\label{spec}
\begin{split}
X_i = x_i,\ a_i = t\ \ (i = 1,\ldots,n);\ \ X_{n+i} = y_i,\ a_{n+i} = q^{-1}\ \ (i = 1,\ldots,m);\\
Y_k = z_k,\ b_k = t\ \ (k = 1,\ldots,N);\ \ Y_{N+k} = w_k,\ b_{N+k} = q^{-1}\ \ (k = 1,\ldots,M).
\end{split}
\end{equation}
Focusing first on the left-hand side of the resulting identity, we note that, due to the presence of the factors
$$
\prod_{i=1}^m \frac{(a_{n+i}X_{n+i}/X_{n+i};q)_{\gamma_{n+i}}}{(qX_{n+i}/X_{n+i};q)_{\gamma_{n+i}}} = \prod_{i=1}^m \frac{(q^{-1};q)_{\gamma_{n+i}}}{(q;q)_{\gamma_{n+i}}},
$$
we only obtain non-zero terms when the components $\gamma_{n+i}$ of $\gamma\in\N^{n+m}$ take the value $0$ or $1$. Hence we may and shall restrict the summation to $n+m$-tuples $\gamma=(\mu,I)$ with $\mu=(\mu_1,\ldots,\mu_n)\in\N^n$ and $I\subseteq\{1,\ldots,m\}$, where, as previously indicated, we identify such a subset $I$ with the $m$-tuple $(I_1,\ldots,I_m)\in\{0,1\}^m$ characterised by $I_i=1$ if and only if $i\in I$. Using the elementary identity $(qa;q)_k/(a;q)_k=(1-q^{\mu_i}a)/(1-a)$, we thus find that the left-hand side of the pertinent identity is given by
\begin{multline*}
\frac{(ut^{-n}q^m;q)_\infty}{(u;q)_\infty}\sum_{\mu\in\N^n}\sum_{I\subseteq\{1,\ldots,m\}} (ut^{-n}q^m)^{|\mu|+|I|} \frac{\Delta(q^\mu x)}{\Delta(x)}\frac{\Delta(q^I y)}{\Delta(y)} \prod_{i=1}^n\prod_{j=1}^m \frac{q^{\mu_i}x_i-q^{I_j}y_j}{x_i-y_j}\\
\cdot \prod_{i=1}^n\Bigg(\prod_{j=1}^n \frac{(tx_i/x_j;q)_{\mu_i}}{(qx_i/x_j;q)_{\mu_i}}\cdot \prod_{j=1}^m \frac{(x_i/qy_j;q)_{\mu_i}}{(qx_i/y_j;q)_{\mu_i}}\Bigg)\cdot \prod_{i\in I}\Bigg(\prod_{j=1}^n \frac{1-ty_i/x_j}{1-qy_i/x_j}\cdot \prod_{j=1}^m \frac{1-y_i/qy_j}{1-qy_i/y_j}\Bigg)\\
\cdot \prod_{i=1}^n\Bigg(\prod_{k=1}^N \frac{(x_iz_k/t;q)_{\mu_i}}{(x_iz_k;q)_{\mu_i}}\cdot \prod_{k=1}^M \frac{1-q^{\mu_i}x_iw_k}{1-x_iw_k}\Bigg)\cdot \prod_{i\in I}\Bigg(\prod_{k=1}^N \frac{1-y_iz_k/t}{1-y_iz_k}\cdot \prod_{k=1}^M \frac{1-qy_iw_k}{1-y_iw_k}\Bigg).
\end{multline*}
We rewrite the factors depending only on $y$,
\begin{equation*}
\begin{split}
\frac{\Delta(q^I y)}{\Delta(y)} \prod_{i\in I}\prod_{j=1}^m \frac{1-y_i/qy_j}{1-qy_i/y_j} &= q^{\binom{|I|}{2}} \prod_{\substack{1\leq i,j\leq m\\ i\in I,j\notin I}} \frac{qy_i-y_j}{y_i-y_j}\cdot \prod_{i\in I} \frac{1-q^{-1}}{1-q}\\& \quad \cdot \prod_{\substack{i,j\in I\\ i<j}} \frac{1-y_i/qy_j}{1-qy_i/y_j}\frac{1-y_j/qy_i}{1-qy_j/y_i}
 \cdot \prod_{\substack{1\leq i,j\leq m\\ i\in I,j\notin I}} \frac{1-y_i/qy_j}{1-qy_i/y_j}\\
 &= (-1)^{|I|}q^{\binom{|I|}{2}-m|I|}\prod_{\substack{1\leq i,j\leq m\\ i\in I,j\notin I}} \frac{y_i-qy_j}{y_i-y_j},
\end{split}
\end{equation*}
and the factors depending on both $x$ and $y$,
\begin{multline*}
\prod_{i=1}^n\prod_{j=1}^m \frac{q^{\mu_i}x_i-q^{I_j}y_j}{x_i-y_j}\cdot \prod_{i=1}^n\prod_{j=1}^m \frac{(x_i/qy_j;q)_{\mu_i}}{(qx_i/y_j;q)_{\mu_i}}\cdot \prod_{i\in I}\prod_{j=1}^n \frac{1-ty_i/x_j}{1-qy_i/x_j}\\
= q^{|I|n} \prod_{i=1}^n \Bigg(\prod_{j\in I} \frac{1-q^{\mu_i-1}x_i/y_j}{1-x_i/y_j}\cdot \prod_{j\notin I} \frac{1-q^{\mu_i}x_i/y_j}{1-x_i/y_j}\Bigg)\\
\cdot \prod_{i=1}^n\prod_{j=1}^m \frac{(x_i/qy_j;q)_{\mu_i}}{(qx_i/y_j;q)_{\mu_i}}\cdot \prod_{i\in I}\prod_{j=1}^n \frac{1-ty_i/x_j}{1-qy_i/x_j}\\
= t^{|I|n} \prod_{i=1}^n \Bigg(\prod_{j\in I} \frac{1-x_i/ty_j}{1-q^{\mu_i}x_i/y_j}\cdot \prod_{j\notin I} \frac{1-x_i/qy_j}{1-q^{\mu_i}x_i/qy_j}\Bigg).
\end{multline*}
In this way, we find that, when specialised at \eqref{spec}, the left-hand side of \eqref{KajForm2} is given by
\begin{multline}
\label{KajFormLHS}
\frac{(ut^{-n}q^m;q)_\infty}{(u;q)_\infty}\sum_{\mu\in\N^n}\sum_{I\subseteq\{1,\ldots,m\}} (ut^{-n}q^m)^{|\mu|}(-u)^{|I|}q^{\binom{|I|}{2}} \frac{\Delta(q^\mu x)}{\Delta(x)}\\
\cdot \prod_{i,j=1}^n \frac{(tx_i/x_j;q)_{\mu_i}}{(qx_i/x_j;q)_{\mu_i}}\cdot \prod_{\substack{1\leq i,j\leq m\\ i\in I,j\notin I}} \frac{y_i-qy_j}{y_i-y_j}\cdot \prod_{i=1}^n \Bigg(\prod_{j\in I} \frac{1-x_i/ty_j}{1-q^{\mu_i}x_i/y_j}\cdot \prod_{j\notin I} \frac{1-x_i/qy_j}{1-q^{\mu_i}x_i/qy_j}\Bigg)\\
\cdot \prod_{i=1}^n\Bigg(\prod_{k=1}^N \frac{(x_iz_k/t;q)_{\mu_i}}{(x_iz_k;q)_{\mu_i}}\cdot \prod_{k=1}^M \frac{1-q^{\mu_i}x_iw_k}{1-x_iw_k}\Bigg)\cdot \prod_{i\in I}\Bigg(\prod_{k=1}^N \frac{1-y_iz_k/t}{1-y_iz_k}\cdot \prod_{k=1}^M \frac{1-qy_iw_k}{1-y_iw_k}\Bigg).
\end{multline}
Using the functional equation $(z;q)_\infty=(1-z)(qz;q)_\infty$, a direct computation yields
\begin{multline*}
\frac{T_{q,x}^\mu T_{t,y}^{-I}\big(\Phi_{n,m;N,M}(x,y;z,w)\big)}{\Phi_{n,m;N,M}(x,y;z,w)} = \prod_{i=1}^n\Bigg(\prod_{k=1}^N \frac{(x_iz_k/t;q)_{\mu_i}}{(x_iz_k;q)_{\mu_i}}\cdot \prod_{k=1}^M \frac{1-q^{\mu_i}x_iw_k}{1-x_iw_k}\Bigg)\\
\cdot \prod_{i\in I}\Bigg(\prod_{k=1}^N \frac{1-y_iz_k/t}{1-y_iz_k}\cdot \prod_{k=1}^M \frac{1-qy_iw_k}{1-y_iw_k}\Bigg).
\end{multline*}
Multiplying this expression with $(t^{-n}q^mu)^{|\mu|} (-u)^{|I|}q^{\binom{|I|}{2}}B_{\mu,I}(x,y)$, we obtain the $(\mu,I)$-term in \eqref{KajFormLHS}. In other words, the specialisation of the left-hand side of \eqref{KajForm2} to \eqref{spec} equals $\Phi_{n,m;N,M}^{-1}\cH_{n,m}(x,y;t^{-1}u)\Phi_{n,m;N,M}$.

We observe that the specialisation of the right-hand side of \eqref{KajForm2} is obtained from its left-hand side by interchanging $(n,m)\leftrightarrow (N,M)$ and $(x,y)\leftrightarrow (z,w)$ as well as relabelling $\mu\to \nu$. Due to the manifest symmetry property
$$
\Phi_{n,m;N,M}(x,y;z,w) = \Phi_{N,M;n,m}(z,w;x,y),
$$
it follows that the right-hand side of \eqref{KajForm2}, when specialised to \eqref{spec}, is given by $\Phi_{n,m;N,M}^{-1}\cH_{N,M}(z,w;t^{-1}u)\Phi_{n,m;N,M}$. This concludes the proof of the kernel identity \eqref{cHkerId}.
\end{proof}

\begin{remark}
By minor modifications of the above proof, we can obtain kernel identities for the parameter regime $|q|,|t|<1$, but only for the deformed NS generating series $\cH_{n,m}$ \eqref{cHnm}. (Indeed, the MR generating series $\cD_{n,m}$ \eqref{cDnm} is well-defined only when $|t|>1$.) More precisely, starting from the identity obtained by taking $x\to tx$ and $y\to ty$ in \eqref{KajForm2} after specialising to \eqref{spec}, it is readily seen that \eqref{cHkerId} holds true for $|q|,|t|<1$ if we replace $\Phi_{n,m;N,M}$ by the meromorphic function
\begin{equation}
\begin{split}
\Psi_{n,m;N,M}(x,y;z,w) &:= \prod_{i=1}^n\prod_{j=1}^N\frac{(tx_iz_j;q)_\infty}{(x_iz_j;q)_\infty}\cdot \prod_{i=1}^m\prod_{j=1}^M\frac{(qty_iw_j;t)_\infty}{(ty_iw_j;t)_\infty}\\
&\quad \cdot \prod_{i=1}^n\prod_{j=1}^M (1-tx_iw_j)\cdot \prod_{i=1}^m\prod_{j=1}^N (1-ty_iz_j).
\end{split}
\end{equation}
\end{remark}

\section{Applications}
\label{Sec:Apps}
In this section, we detail a number of applications of Theorem \ref{Thm:kerIds}. They include the commutativity of the deformed NS and MR operators; a derivation of their joint eigenfunctions and eigenvalues; an explicit construction of a Harish-Chandra type isomorphism, characterising the commutative algebra generated by the deformed NS (and/or MR) operators; as well as a generalisation of the restriction picture for the first order operators in \cite{SV09} to all higher order operators.

We note that intermediate computations, involving kernel functions and generating series, may require restrictions on $q,t$ of the form $|q|<1$ and/or $|t|>1$. To ease the exposition, we shall not spell out the specific restrictions that are needed whenever they are easily identified from the context at hand.

\subsection{Commutativity}
\label{SubSec:Comm}
We find it convenient to work with the difference operators $\cH_{n,m}^{r}(x,y)$ ($r\in\N$) and $\cD_{n,m}^{r}(x,y)$ ($r\in\N$) defined as the coefficients of $u^r$ in the power series expansion of $\cH_{n,m}(x,y;u)$ \eqref{cHnm} and $\cD_{n,m}(x,y;u)$ \eqref{cDnm}, respectively:
\begin{equation}
\label{cHnmk}
\cH_{n,m}(x,y;u) = \sum_{r=0}^\infty u^r \cH_{n,m}^{r}(x,y),
\end{equation}
and
\begin{equation}
\label{cDnmk}
\cD_{n,m}(x,y;u) = \sum_{r=0}^\infty (-u)^r \cD_{n,m}^{r}(x,y).
\end{equation}

We begin by recording the following important technical result.

\begin{lemma}
\label{Lemma:diffop}
Let $L_{n,m}(x,y)$ be a difference operator in $(x,y)$ of the form
$$
L_{n,m}(x,y) = \sum_{\substack{\mu\in\N^n, \nu\in\N^m\\ |\mu|+|\nu|\leq d}}a_{\mu,\nu}(x,y) T_{q,x}^\mu T_{t,y}^{-\nu},
$$
with meromorphic coefficients $a_{\mu,\nu}(x,y)$ and $d\in\N$. If $L_{n,m}(x,y)\Phi_{n,m;N,0}(x,y;z)=0$ for all $N\in\N^*$, then $L_{n,m}(x,y)\equiv 0$ as a difference operator.
\end{lemma}

\begin{proof}
The proof is given in Appendix \ref{App:diffop}.
\end{proof}

Comparing \eqref{Hn} with \eqref{Hnm} and \eqref{cHnm}, we see that
\begin{equation*}
\begin{split}
\cH_{N,0}(z;u) &= \frac{(t^{1-N}u;q)_\infty}{(tu;q)_\infty}H_N(z;t^{1-N}u)\\
&= \frac{(t^{1-N}u;q)_\infty}{(tu;q)_\infty}\sum_{r=0}^\infty \big(t^{1-N}u\big)^r H_N^r(z).
\end{split}
\end{equation*}
From the commutativity of the NS operators $H_N^r$, we thus get
$$
\left[\cH_{N,0}(z;u),\cH_{N,0}(z;v)\right] = 0.
$$
Taking $M=0$ in \eqref{cHkerId}, we can now deduce
\begin{multline*}
\cH_{n,m}(x,y;u)\cH_{n,m}(x,y;v)\Phi_{n,m;N,0}(x,y;z)\\
= \cH_{N,0}(z;v)\cH_{N,0}(z;u)\Phi_{n,m;N,0}(x,y;z)\\
= \cH_{N,0}(z;u)\cH_{N,0}(z;v)\Phi_{n,m;N,0}(x,y;z)\\
= \cH_{n,m}(x,y;v)\cH_{n,m}(x,y;u)\Phi_{n,m;N,0}(x,y;z),
\end{multline*}
so that
$$
\left[\cH_{n,m}(x,y;u),\cH_{n,m}(x,y;v)\right]\Phi_{n,m;N,0}(x,y;z) = 0,
$$
or equivalently
$$
\left[\cH_{n,m}^{r}(x,y),\cH_{n,m}^{s}(x,y)\right]\Phi_{n,m;N,0}(x,y;z) = 0\ \ \ (r,s\in\N).
$$
Hence, fixing $r,s\in\N$ and letting
$$
L_{n,m}(x,y) = \left[\cH_{n,m}^{r}(x,y),\cH_{n,m}^{s}(x,y)\right],
$$
we have $L_{n,m}(x,y)\Phi_{n,m;N,0}(x,y;z)=0$ for all $N\in\N^*$. It follows from Lemma \ref{Lemma:diffop} that $L_{n,m}(x,y)\equiv 0$, i.e.~that $\cH_{n,m}^{r}$ and $\cH_{n,m}^{s}$ commute as difference operators.

Repeating the above reasoning with either one or both of $\cH_{n,m}(x,y;u)$ and $\cH_{n,m}(x,y;v)$ replaced by $\cD_{n,m}(x,y;u)$ and $\cD_{n,m}(x,y;v)$, respectively, we arrive at the following result.

\begin{theorem}
The deformed NS operators $\cH_{n,m}^{r}$ ($r\in\N$) and MR operators $\cD_{n,m}^{r}$ ($r\in\N$) all commute with each other:
$$
\left[\cH_{n,m}^{r},\cH_{n,m}^{s}\right] = \left[\cD_{n,m}^{r},\cD_{n,m}^{s}\right] = \left[\cH_{n,m}^{r},\cD_{n,m}^{s}\right] = 0\ \ \ (r,s\in\N).
$$
\end{theorem}

\subsection{Joint eigenfunctions}
\label{SubSec:EigFuncs}
Next, we show that the deformed MR and NS operators are simultaneously diagonalised by the so-called super Macdonald polynomials, introduced in \cite{SV09} as certain restrictions of Macdonald symmetric functions. Here we pursue a somewhat different approach: From the kernel identities in Theorem \ref{Thm:kerIds} and well-known results on ordinary Macdonald polynomials, we recover an expression for the super Macdonald polynomials in terms of the ordinary Macdonald polynomials and deduce corresponding eigenvalue equations with explicit expressions for the eigenvalues.

For notation and terminology regarding symmetric functions in general and Macdonald symmetric functions (and polynomials) in particular, we follow Macdonald's book \cite{Mac95}.

Unless otherwise specified, we assume throughout this and the following sections that
\begin{equation}
\label{nonSpec}
q^it^j\neq 1\ \text{for all}\ i,j\in\N \ \text{such that}\ i+j\geq 1,
\end{equation}
which, in particular, ensures that the Macdonald functions are well-defined. Following Sergeev and Veselov \cite{SV09}, we use the terminology {\it non-special} for values of $q,t\in\C^*$ satisfying \eqref{nonSpec}.

Setting $M=0$ in the kernel function \eqref{kerFunc}, we define polynomials $SP_\lambda(x,y)=SP_\lambda(x,y;q,t)$ as the appropriately scaled coefficients of the (dual) Macdonald polynomials $Q_\lambda(z)=Q_\lambda(z;q,t)$ in its power series expansion in the variables $z_1,\ldots,z_N$:
\begin{equation}
\label{SPDef}
\Phi_{n,m;N,0}(x,y;z) = \sum_{l(\lambda)\leq N} t^{-|\lambda|}SP_\lambda(x,y)Q_\lambda(z).
\end{equation}
Assuming $N\geq n$, we recall from Sections VI.4--5 in \cite{Mac95} that
$$
\prod_{i=1}^n\prod_{j=1}^N \frac{(x_iz_j;q)_\infty}{(t^{-1}x_iz_j;q)_\infty} = \sum_{l(\mu)\leq n} t^{-|\mu|}P_\mu(x)Q_\mu(z),
$$
$$
\prod_{i=1}^m\prod_{j=1}^N (1-y_iz_j) = \sum_{\nu\subseteq(m^N)} (-1)^{|\nu|} Q_{\nu^\prime}(y;t,q)Q_\nu(z;q,t).
$$
It follows that
\begin{multline*}
\Phi_{n,m;N,0}(x,y;z;q,t)\\
= \sum_{l(\mu)\leq n}\sum_{\nu\subseteq(m^N)} t^{-|\mu|}(-1)^{|\nu|}P_\mu(x;q,t)Q_{\nu^\prime}(y;t,q)Q_\mu(z;q,t)Q_\nu(z;q,t).
\end{multline*}
Letting $\hat{c}_{\mu,\nu}^\lambda(q,t)$ denote the Littlewood--Richardson type coefficients for $Q_\lambda(z;q,t)$,
\begin{equation}
\label{LRCoeffs}
Q_\mu(z;q,t)Q_\nu(z;q,t) = \sum_{l(\lambda)\leq N} \hat{c}_{\mu,\nu}^\lambda(q,t) Q_\lambda(z;q,t),
\end{equation}
we get
\begin{multline*}
\Phi_{n,m;N,0}(x,y;z;q,t)\\
= \sum_{l(\lambda)\leq N}t^{-|\lambda|}\Bigg(\sum_{l(\mu)\leq n}\sum_{\nu\subseteq(m^N)} (-t)^{|\nu|}\hat{c}_{\mu,\nu}^\lambda(q,t) P_\mu(x;q,t)Q_{\nu^\prime}(y;t,q)\Bigg) Q_\lambda(z;q,t),
\end{multline*}
where we have used the fact that $\hat{c}_{\mu,\nu}^\lambda(q,t)\neq 0$ only if $|\lambda|=|\mu|+|\nu|$, which is a direct consequence of $Q_\lambda(z)$ being a homogeneous polynomial of degree $|\lambda|$.

Since $\hat{c}_{\mu,\nu}^\lambda(q,t)=0$ unless $\mu,\nu\subseteq\lambda$ (cf.~Section VI.7 in \cite{Mac95}), $l(\lambda)\leq N$ and $Q_{\nu^\prime}((y_1,\ldots,y_m);t,q)\equiv 0$ if $\nu_1>m$, we can replace the summation criterion $\nu\subseteq(m^N)$ by $\nu\subseteq\lambda$, say. Comparing the resulting expansion with \eqref{SPDef}, we see that
\begin{equation}
\label{SPExp1}
SP_\lambda(x,y;q,t) = \sum_{l(\mu)\leq n}\sum_{\nu\subseteq\lambda} (-t)^{|\nu|}\hat{c}_{\mu,\nu}^\lambda(q,t) P_\mu(x;q,t)Q_{\nu^\prime}(y;t,q),
\end{equation}
where $\lambda$ can be any partition, since $N(\geq n)$ can be chosen arbitrarily large.

Using the skew Macdonald polynomials
$$
P_{\lambda/\nu}(x;q,t) = \sum_\mu \hat{c}_{\mu,\nu}^\lambda(q,t) P_\mu(x;q,t),
$$
we can rewrite this expression as
\begin{equation}
\label{SPExp2}
SP_\lambda(x,y;q,t) = \sum_{\nu\subseteq\lambda} (-t)^{|\nu|}P_{\lambda/\nu}(x;q,t)Q_{\nu^\prime}(y;t,q).
\end{equation}

A direct comparison with Eq.~(22) in Sergeev and Veselov's paper \cite{SV09} reveals that these polynomials are precisely the so-called super Macdonald polynomials, as defined by Eq.~(23) in {\it loc.~cit.}. (Note their use of the inverse $t^{-1}$ of the parameter $t$ used here and that $H(\lambda,q,t)/H(\lambda^\prime,t,q)=(-t)^{-|\lambda|}b_{\lambda^\prime}(t^{-1},q)$, cf.~the equation above (6.19) in Chapter VI of \cite{Mac95}.)

In analogy with Macdonald's definition of $Q_\lambda$, we let
$$
SQ_\lambda(x,y) = b_\lambda SP_\lambda(x,y)
$$
with
\begin{equation}
\label{blam}
b_\lambda = \prod_{s\in\lambda} \frac{1-q^{a(s)}t^{l(s)+1}}{1-q^{a(s)+1}t^{l(s)}},
\end{equation}
where $a(s)=\lambda_i-j$ and $l(s)=\lambda^\prime_j-i$ denote the arm- and leg length of $s=(i,j)\in\lambda$ respectively.

In the following proposition, we record two symmetry properties of the super Macdonald polynomials that we have occasion to invoke below.

\begin{proposition}
For all $\lambda\in H_{n,m}$ and non-special $q,t$, we have
\begin{equation}
\label{SPSym1}
SP_\lambda(x,y;q^{-1},t^{-1}) = SP_\lambda(x,q^{-1}t^{-1}y;q,t),
\end{equation}
\begin{equation}
\label{SPSym2}
SQ_{\lambda^\prime}(y,x;t^{-1},q^{-1}) = (-q)^{-|\lambda|}SP_\lambda(x,y;q,t).
\end{equation}
\end{proposition}

\begin{proof}
From (4.14)(iv) in Chapter VI of \cite{Mac95}, we recall
\begin{equation}
\label{PQIds}
P_\lambda(x;q^{-1},t^{-1}) = P_\lambda(x;q,t),\ \ \ Q_\lambda(x;q^{-1},t^{-1}) = (qt^{-1})^{|\lambda|}Q_\lambda(x;q,t);
\end{equation}
and using (7.3) in {\it loc.~cit.}, we thus infer
\begin{equation}
\label{hatcId1}
\hat{c}_{\mu,\nu}^\lambda(q^{-1},t^{-1}) = \hat{c}_{\mu,\nu}^\lambda(q,t).
\end{equation}
\begin{equation}
\label{hatcId2}
\hat{c}_{\mu^\prime,\nu^\prime}^{\lambda^\prime}(t^{-1},q^{-1}) = \hat{c}_{\mu^\prime,\nu^\prime}^{\lambda^\prime}(t,q) = \hat{c}_{\mu,\nu}^\lambda(q,t)b_\lambda(q,t)/b_\mu(q,t)b_\nu(q,t).
\end{equation}

Keeping \eqref{SPExp1} in mind, we see that \eqref{SPSym1} is a simple consequence of \eqref{PQIds}, \eqref{hatcId1} and the fact that $Q_{\nu^\prime}(y)$ is a homogeneous polynomial of degree $|\mu|$. Furthermore, appealing to \eqref{PQIds} as well as \eqref{hatcId2}, we deduce
\begin{multline*}
SQ_{\lambda^\prime}(y,x;t^{-1},q^{-1})\\
= b_{\lambda^\prime}(t^{-1},q^{-1})b_\lambda(q,t)\sum_{\mu,\nu}(-t)^{-|\nu|}\hat{c}_{\mu,\nu}^\lambda(q,t)P_\nu(x;q,t)\frac{P_{\mu^\prime}(y;t,q)}{b_\mu(q,t)}.
\end{multline*}
Utilising \eqref{blam}, it is readily seen that $b_{\lambda^\prime}(t^{-1},q^{-1})=(q^{-1}t)^{|\lambda|}/b_\lambda(q,t)$ and $b_\mu(q,t)=1/b_{\mu^\prime}(t,q)$, which clearly entails \eqref{SPSym2}.
\end{proof}

We recall that, by analysing \eqref{SPExp2}, Sergeev and Veselov showed that $SP_\lambda(x,y)$ vanishes identically unless $\lambda$ is contained in the set of partitions $H_{n,m}$, consisting of all partitions $\lambda$ such that $\lambda_{n+1}\leq m$, or equivalently, the diagram of $\lambda$ is contained in the so-called fat $(n,m)$-hook; and the non-zero super Macdonald polynomials $SP_\lambda(x,y)$ ($\lambda\in H_{n,m}$) form a basis in $\Lambda_{n,m;q,t}$, the algebra of (complex) polynomials $p(x,y)$ in $n+m$ variables $x=(x_1,\ldots,x_n)$ and $y=(y_1,\ldots,y_m)$ that are symmetric in each set of variables separately,
$$
p(\sigma x,\tau y) = p(x,y)\ \ \ ((\sigma,\tau)\in S_n\times S_m),
$$
and satisfy the additional symmetry conditions
\begin{equation}
\label{qinv}
\big(T_{q,x_i}-T_{t,y_j}^{-1}\big)p(x,y) = 0\ \ \text{along}\ \ x_i = y_j\ \ \ (i=1,\ldots,n,\ j=1,\ldots,m),
\end{equation}
cf.~Thm.~5.6 in \cite{SV09}.

To establish the desired eigenvalue equations, we focus first on the deformed NS operators.  Specifically, taking $M=0$ in \eqref{cHkerId}, we obtain
\begin{equation}
\label{cHkerIdspec}
\cH_{n,m}(x,y;u)\Phi_{m,m;N,0}(x,y;z) = \frac{(t^{1-N}u;q)_\infty}{(tu;q)_\infty}H_N(z;t^{1-N}u)\Phi_{n,m;N,0}(x,y;z),
\end{equation}
and from Eq.~(5.17) in \cite{NS20}, we infer
\begin{equation*}
H_N(z;t^{1-N}u)Q_\lambda(z) = Q_\lambda(z) \prod_{i=1}^N \frac{(q^{\lambda_i}t^{2-i}u;q)_\infty}{(q^{\lambda_i}t^{1-i}u;q)_\infty}\ \ \ (l(\lambda)\leq N).
\end{equation*}
For $\lambda\in H_{n,m}$, we introduce the product
\begin{equation}
\label{cGlam}
\mathcal{G}_\lambda(u) = \prod_{i\geq 1} \frac{(q^{\lambda_i}t^{2-i}u;q)_\infty}{(t^{2-i}u;q)_\infty}\frac{(t^{1-i}u;q)_\infty}{(q^{\lambda_i}t^{1-i}u;q)_\infty},
\end{equation}
(which may be truncated at $i=l(\lambda)$). Choosing $N\geq l(\lambda)$ and substituting the expansion \eqref{SPDef} in the kernel identity \eqref{cHkerIdspec},  we deduce
\begin{equation}
\label{cHnmEigEq}
\cH_{n,m}(x,y;u)SP_\lambda(x,y) = \mathcal{G}_\lambda(u)SP_\lambda(x,y).
\end{equation}

Rather than expressing the eigenvalue $\mathcal{G}_\lambda(u)$ in terms of the quantities $q^{\lambda_i}$ ($i\geq 1$), it is in many ways more natural to map $\lambda\in H_{n,m}$ (injectively) to the pair of partitions
\begin{equation}
\label{munu}
\mu = (\lambda_1,\ldots,\lambda_n),\ \ \ \nu = (\lambda_{n+1},\lambda_{n+2},\ldots)^\prime,
\end{equation}
and rewrite \eqref{cGlam} in terms of $q^{\mu_i}$ ($i = 1,\ldots,n$) and $t^{-\nu_j-n}$ ($j = 1,\ldots,m$). More precisely, we have the equalities
\begin{align*}
\prod_{i\geq n+1} \frac{(q^{\lambda_i}t^{2-i}u;q)_\infty}{(t^{2-i}u;q)_\infty}\frac{(t^{1-i}u;q)_\infty}{(q^{\lambda_i}t^{1-i}u;q)_\infty} &= \prod_{i\geq n+1}\prod_{j=1}^{\lambda_i} \frac{1-q^{j-1}t^{1-i}u}{1-q^{j-1}t^{2-i}u}\\
&= \prod_{j=1}^m\prod_{i=1}^{\nu_j} \frac{1-q^{j-1}t^{1-i-n}u}{1-q^{j-1}t^{2-i-n}u}\\
&= \prod_{j=1}^m \frac{1-q^{j-1}t^{1-\nu_j-n}u}{1-q^{j-1}t^{1-n}u},
\end{align*}
which entail
\begin{equation}
\label{cGExpr}
\mathcal{G}_\lambda(u) = \prod_{i=1}^n \frac{(q^{\mu_i}t^{2-i}u;q)_\infty}{(t^{2-i}u;q)_\infty}\frac{(t^{1-i}u;q)_\infty}{(q^{\mu_i}t^{1-i}u;q)_\infty}\cdot \prod_{j=1}^m \frac{1-t^{1-\nu_j-n}q^{j-1}u}{1-t^{1-n}q^{j-1}u}.
\end{equation}
We note that the right-hand side is manifestly invariant under permutations of the quantities $q^{\mu_i}t^{1-i}$ ($i=1,\ldots,n$) as well as the quantities $t^{-\nu_j-n}q^{j-1}$ ($j = 1,\ldots,m$).

Substituting $q^{\mu_i}\to x_i$ ($i=1,\ldots,n$) and $t^{-\nu_j-n}\to y_j$ ($j=1,\ldots,m$) in \eqref{cGExpr}, we obtain the product function
\begin{equation}
\label{Gnat}
G^\natural_{n,m}(x,y;u) = \prod_{i=1}^n \frac{(x_it^{2-i}u;q)_\infty}{(t^{2-i}u;q)_\infty}\frac{(t^{1-i}u;q)_\infty}{(x_it^{1-i}u;q)_\infty}\cdot \prod_{j=1}^m \frac{1-ty_jq^{j-1}u}{1-t^{1-n}q^{j-1}u},
\end{equation}
so that
\begin{equation}
\label{cGlamExpr2}
\mathcal{G}_\lambda(u)=G^\natural_{n,m}(q^\mu,t^{-\nu-(n^m)};u).
\end{equation}
If we now define polynomials $g_r^\natural(x,y)=g_r^\natural(x,y;q,t)$ ($r\in\N$) as the coefficients of $u^r$ in the power series expansion of \eqref{Gnat}, i.e.
\begin{equation}
\label{gknat}
G^\natural_{n,m}(x,y;u) = \sum_{r\geq 0}g_r^\natural(x,y)u^r,
\end{equation}
then it becomes clear from \eqref{cHnmk}, \eqref{cHnmEigEq} and \eqref{cGlamExpr2} that the eigenvalues of $\cH_{n,m}^{r}(x,y)$ are given by $g_r^\natural(q^\mu,t^{-\nu-(n^m)})$.

The eigenvalues of the deformed MR operators in \eqref{cDnmk} can, in a similar manner, be expressed in terms of polynomials $e_r^\natural(x,y)$ ($r\in\N$) defined by the generating function expansion
\begin{equation}
\label{Enat}
\begin{split}
E^\natural_{n,m}(x,y;u) &= \prod_{i=1}^n \frac{1-x_it^{1-i}u}{1-t^{1-i}u} \cdot \prod_{j=1}^m \frac{(t^{-n}q^{j}u;t^{-1})_\infty}{(y_jq^{j}u;t^{-1})_\infty}\frac{(y_jq^{j-1}u;t^{-1})_\infty}{(t^{-n}q^{j-1}u;t^{-1})_\infty}\\
&= \sum_{r\geq 0}e_r^\natural(x,y;q,t)(-u)^r.
\end{split}
\end{equation}
Indeed, we have the following theorem, which details the explicit simultaneous diagonalisation of the deformed NS and MR operators.

\begin{theorem}
\label{Thm:EigEqs}
Assuming that $q,t$ are non-special, we have the eigenvalue equations
\begin{equation}
\label{cHnmkEigEq}
\cH_{n,m}^{r}(x,y)SP_\lambda(x,y) = g_r^\natural(q^\mu,t^{-\nu-(n^m)})SP_\lambda(x,y)
\end{equation}
and
\begin{equation}
\label{cDnmkEigEq}
\cD_{n,m}^{r}(x,y)SP_\lambda(x,y) = e_r^\natural(q^\mu,t^{-\nu-(n^m)})SP_\lambda(x,y)
\end{equation}
for all $r\in\N$, $\lambda\in H_{n,m}$ and with $\mu,\nu$ given by \eqref{munu}.
\end{theorem}

\begin{proof}
There remains only to establish the latter eigenvalue equation.

Letting
$$
\eta = (\lambda^\prime_1,\ldots,\lambda^\prime_m),\ \ \ \xi = (\lambda^\prime_{m+1},\lambda^\prime_{m+2},\ldots)^\prime,
$$
we use $\cD_{n,m}(x,y;u;q,t)=\cH_{m,n}(y,x;qu;t^{-1},q^{-1})$ and the symmetry property \eqref{SPSym2} of $SP_\lambda$ to infer from \eqref{cHnmk}, \eqref{gknat} and \eqref{cHnmkEigEq} that
$$
\cD_{n,m}(x,y;u;q,t)SP_\lambda(x,y;q,t) = G^\natural_{m,n}(t^{-\eta},q^{\xi+(m^n)};qu;t^{-1},q^{-1})SP_\lambda(x,y;q,t).
$$
Hence \eqref{cDnmkEigEq} will follow once we prove that
\begin{equation}
\label{GEId}
G^\natural_{m,n}(t^{-\eta},q^{\xi+(m^n)};qu;t^{-1},q^{-1}) = E^\natural_{n,m}(q^\mu,t^{-\nu-(n^m)};u;q,t).
\end{equation}
By a direct computation, similar to that leading from \eqref{cGlam} to \eqref{cGExpr}, it is readily verified that
\begin{equation}
\label{EnatExpr}
E^\natural_{n,m}(q^\mu,t^{-\nu-(n^m)};u) = \prod_{i\geq 1} \frac{1-q^{\lambda_i}t^{1-i}u}{1-t^{1-i}u},
\end{equation}
as long as $\lambda\in H_{n,m}$. Keeping in mind \eqref{cGlam} and \eqref{cGlamExpr2}, it becomes clear that both the left- and right-hand side of \eqref{GEId} are independent of $n,m$, so that we may choose them such that $(n^m)\subseteq\lambda$. As a consequence, we get
$$
\eta = \nu+(n^m),\ \ \ \xi = \mu-(m^n),
$$
which entails
$$
G^\natural_{m,n}(t^{-\eta},q^{\xi+(m^n)};qu;t^{-1},q^{-1}) = G^\natural_{m,n}(t^{-\nu+(n^m)},q^{\mu};qu;t^{-1},q^{-1}).
$$
Observing
\begin{multline*}
\prod_{i=1}^m \frac{(q^iu;t^{-1})_\infty}{(q^{i-1}u;t^{-1})_\infty}\cdot \prod_{j=1}^n \frac{1}{1-q^mt^{1-j}u}\\
= \frac{(q^mt^{-n}u;t^{-1})_\infty}{(u;t^{-1})_\infty} = \prod_{i=1}^n \frac{1}{1-t^{1-i}u}\cdot \prod_{j=1}^m \frac{(t^{-n}q^ju;t^{-1})_\infty}{(t^{-n}q^{j-1}u;t^{-1})_\infty},
\end{multline*}
we deduce
$$
G^\natural_{m,n}(y,x;qu;t^{-1},q^{-1}) = E^\natural_{n,m}(x,y;u;q,t)
$$
and \eqref{GEId} clearly follows.
\end{proof}

Introducing the difference operators
$$
\widehat{\cH}_{n,m}^{r}(x,y;q,t) := \cH_{n,m}^{r}(x,qty;q^{-1},t^{-1})
$$
and
$$
\widehat{\cD}_{n,m}^{r}(x,y;q,t) := \cD_{n,m}^{r}(x,qty;q^{-1},t^{-1})
$$
for $r\in\N$, the following eigenvalue equations are a direct consequence of Theorem \ref{Thm:EigEqs} and symmetry property \eqref{SPSym1} of $SP_\lambda(x,y)$.

\begin{corollary}
\label{Cor:EigEqs}
For all $q,t$ that are non-special, $r\in\N$ and $\lambda\in H_{n,m}$, we have
$$
\widehat{\cH}_{n,m}^{r}(x,y)SP_\lambda(x,y) = g_r^\natural(q^{-\mu},t^{\nu+(n^m)})SP_\lambda(x,y)
$$
and
$$
\widehat{\cD}_{n,m}^{r}(x,y)SP_\lambda(x,y) = e_r^\natural(q^{-\mu},t^{\nu+(n^m)})SP_\lambda(x,y).
$$
\end{corollary}

Taking $r,s\in\N$, let us consider the difference operator
$$
L_{n,m}(x,y) := T_{q,x}^{(r^n)}T_{t,y}^{-(1^m)}\left[\widehat{\cH}_{n,m}^{r},\cH_{n,m}^{s}\right],
$$
to which Lemma \ref{Lemma:diffop} clearly applies. Combining the kernel function expansion \eqref{SPDef} with Theorem \ref{Thm:EigEqs} and Corollary \ref{Cor:EigEqs}, we see that $L_{n,m}(x,y)\Phi_{n,m;N,0}(x,y;z)=0$ for all $N\in\N^*$. By invoking Lemma \ref{Lemma:diffop} and using the invertibility of $T_{q,x}^{(r^n)}T_{t,y}^{-(1^m)}$, we thus conclude that $\widehat{\cH}_{n,m}^{r}$ and $\cH_{n,m}^{s}$ commute as difference operators. Substituting $\widehat{\cH}_{n,m}^{r}\to \widehat{\cD}_{n,m}^{r}$ and/or $\cH_{n,m}^{s}\to \cD_{n,m}^{s}$ in the above argument, we obtain the following corollary.

\begin{corollary}
The difference operators $\widehat{\cH}_{n,m}^{r}$ ($r\in\N$) and $\widehat{\cD}_{n,m}^{r}$ ($r\in\N$) commute with each other as well as the difference operators $\cH_{n,m}^{s}$ ($s\in\N$) and $\cD_{n,m}^{s}$ ($s\in\N$).
\end{corollary}

\subsection{Harish-Chandra isomorphism}
\label{SubSec:HC}
Focusing first on the deformed NS operators, we consider the commutative (complex) algebra of difference operators
\begin{equation}
\label{cRnm}
\cR_{n,m;q,t} := \C\left[\cH_{n,m}^{1},\cH_{n,m}^{2},\ldots\right],
\end{equation}
cf.~ \eqref{cHnmk}. As we demonstrate below, Theorem \ref{Thm:EigEqs} enables us to establish an explicit Harish-Chandra type isomorphism $\Lambda^{\natural}_{n,m;q,t}\to \cR_{n,m;q,t}$, where $\Lambda^{\natural}_{n,m;q,t}$, introduced in \cite{SV09} as a `shifted' version of $\Lambda_{n,m;q,t}$, denotes the algebra of (complex) polynomials $p(x,y)$ in $n+m$ variables $x=(x_1,\ldots,x_n)$ and $y=(y_1,\ldots,y_m)$ that are separately symmetric in the $t$-shifted variables $x_1,x_2t^{-1},\ldots,x_nt^{1-n}$ and the $q$-shifted variables $y_1,y_2q,\ldots,y_mq^{m-1}$, and, in addition, satisfy the symmetry conditions
\begin{equation}
\label{symCondShift}
T_{q,x_i}(p) = T_{t,y_j}^{-1}(p)\ \ \text{along}\ \ x_it^{1-i}=y_jq^{j-1}\ \ \ (i=1,\ldots,n,\ j=1,\ldots,m).
\end{equation}

In particular, the algebra $\Lambda^{\natural}_{n,m;q,t}$ contains the polynomials $g_r^\natural(x,y)$ ($r\in\N$), as defined by \eqref{gknat}. Indeed, their generating function $G^\natural_{n,m}(x,y;u)$ is manifestly symmetric in the shifted variables $x_it^{1-i}$ and $y_jq^{j-1}$ and, by a direct computation, it is readily seen that $G^\natural_{n,m}(x,y;u)$ satisfies \eqref{symCondShift} as well. Furthermore, using corresponding elements in the so-called algebra of shifted symmetric functions, we can prove the following result.

\begin{lemma}
\label{Lemma:gknat}
As long as the parameters $q,t$ are non-special, the algebra $\Lambda^{\natural}_{n,m;q,t}$ is generated by the polynomials $g_r^\natural(x,y)$ ($r\in\N^*$).
\end{lemma}

\begin{proof}
The proof of this lemma is relegated to Appendix \ref{App:gknat}.
\end{proof}

We note that, as $\lambda$ runs through all partitions in the fat hook $H_{n,m}$, the corresponding points $(q^\mu,t^{-\nu-(n^m)})$, with $\mu,\nu$ given by \eqref{munu}, form a Zariski-dense set in $\C^{n+m}$, i.e.~the only polynomial $p(x,y)$ in $n+m$ variables $(x,y)$ that vanishes at all these points is the zero polynomial. Combining this observation with Lemmas \ref{Lemma:diffop} and \ref{Lemma:gknat}, it is now straightforward to establish the Harish-Chandra isomorphism. 

\begin{theorem}
\label{Thm:HCiso}
For non-special $q,t$, the map
$$
g_r^\natural((x_1,\ldots,x_n),(y_1,\ldots,y_m))\mapsto \cH_{n,m}^{r}\ \ \ (r\in\N^*)
$$
extends to an isomorphism
\begin{equation}
\label{psi}
\psi: \Lambda_{n,m;q,t}^\natural\to \cR_{n,m;q,t}, f\mapsto \cH_{n,m}^f
\end{equation}
of algebras, which is characterised by the eigenfunction property
$$
\cH_{n,m}^f SP_\lambda = f(q^\mu,t^{-\nu-(n^m)})SP_\lambda\ \ \ (f\in\Lambda_{n,m;q,t}^\natural,\ \lambda\in H_{n,m}).
$$
\end{theorem}

\begin{proof}
Suppose we have a relation $F(g_{r_1}^\natural,\ldots,g_{r_K}^\natural)\equiv 0$ for some $F\in\C[z_1,\ldots,z_K]$, with $K\in\N^*$, and $r_j\in\N^*$ for $1\leq j\leq K$. Then, by \eqref{SPDef} and \eqref{cHnmkEigEq}, the corresponding difference operator
$$
L_{n,m} := F\big(\cH_{n,m}^{r_1},\ldots,\cH_{n,m}^{r_K}\big)
$$
satisfies $L_{n,m}\Phi_{n,m;N,0}(x,y;z)=0$ for all $N\in\N^*$. Thanks to Lemma \ref{Lemma:diffop}, it follows that $L_{n,m}\equiv 0$ as a difference operator. Since the polynomials $g^\natural_r$ ($r\in\N^*$) generate $\Lambda^\natural_{n,m;q,t}$ (cf.~Lemma \ref{Lemma:gknat}), we can thus conclude that $\psi$ is a well defined homomorphism of algebras and, as such, it is clearly surjective. To establish injectivity, it suffices to note that $\cH_{n,m}^f\equiv 0$ implies that $f(q^\mu,t^{-\nu-(n^m)})=0$ for all partitions $\mu,\nu$ of the form \eqref{munu} for some $\lambda\in H_{n,m}$, which, as previously observed, entails that $f$ vanishes identically.
\end{proof}

\begin{remark}
This yields an explicit realisation of the monomorphism $\psi$ in Thm.~6.4 of \cite{SV09}.
\end{remark}

By a direct computation, it is readily verified that the generating functions $G_{n,m}^\natural(x,y;u)$ \eqref{Gnat} and $E_{n,m}^\natural(x,y;u)$ \eqref{Enat} satisfy the functional equation
$$
E_{n,m}^\natural(x,y;u)G_{n,m}^\natural(x,y;u) = E_{n,m}^\natural(x,y;tu)G_{n,m}^\natural(x,y;qu).
$$
In view of \eqref{gknat}--\eqref{Enat}, this equation is equivalent to the Wronski type recurrence relations
$$
\sum_{r+s=k} (-1)^r (1-t^rq^s)e_r^\natural(x,y)g_s^\natural(x,y) = 0\ \ \ (k\in\N^*).
$$
Applying the Harish-Chandra isomorphism $\psi$ \eqref{psi}, we obtain the following corollary.

\begin{corollary}
\label{Cor:recRels}
For $q,t$ non-special, the deformed MR and NS operators satisfy the recurrence relations
\begin{equation}
\label{DHRels}
\sum_{r+s=k} (-1)^r (1-t^rq^s)\cD_{n,m}^{r}\cH_{n,m}^{s} = 0\ \ \ (k\in\N^*).
\end{equation}
\end{corollary}

These recurrence relations enable us to express the deformed MR operators $\cD_{n,m}^{r}$ in terms of the deformed NS operators $\cH_{n,m}^{r}$, and vice versa. We can thus conclude that the former operators generate the same commutative algebra as the latter.

\begin{corollary}
Under the assumption that $q,t$ are non-special, we have
$$
\mathcal{R}_{n,m;q,t} = \C\left[\cD_{n,m}^{1},\cD_{n,m}^{2},\ldots\right].
$$
\end{corollary}

Moreover, since the recurrence relations \eqref{DHRels} are of precisely the same form as in the undeformed case (cf.~Eq.~(5.5) in \cite{NS20}), the explicit (determinantal) relations in \cite{NS20} between the undeformed MR and NS operators carry over to the deformed case with minimal (and obvious) changes.

In \cite{FS14} (see Thm.~4.5), Feigin and Silantyev proved that the deformed MR operators $D_{n,m}^r$ with $r=1,\ldots,n+m$ are algebraically independent and thus define an integrable system. As a further application of Theorem \ref{Thm:HCiso}, we give a new proof of this fact.

\begin{corollary}
As long as $q,t$ are non-special, the algebra of difference operators $\cR_{n,m;q,t}$ contains $n+m$ algebraically independent elements, namely the deformed NS operators $\cH_{n,m}^{1},\ldots,\cH_{n,m}^{n+m}$ as well as the deformed MR operators $\cD_{n,m}^{1},\ldots,\cD_{n,m}^{n+m}$.
\end{corollary}

\begin{proof}
Since $\cH_{n,m}(x,y;u;q,t)=\cD_{m,n}(y,x;tu;t^{-1},q^{-1})$ (cf.~\eqref{Hnm}--\eqref{cDnm}), it suffices to prove the claim for the first $n+m$ deformed MR operators, which, thanks to Theorems \ref{Thm:EigEqs} and \ref{Thm:HCiso}, is equivalent to the polynomials $e_r^\natural(x,y)$ with $r=1,\ldots,n+m$ being algebraically independent.

To this end, we recall the deformed shifted power sums $p_r^\natural(x,y)=p_r^\natural(x,y;q,t)$ \eqref{pknat} and introduce their generating function
$$
P^\natural_{n,m}(x,y;u) := \sum_{r=1}^\infty p_r^\natural(x,y)u^{r-1}.
$$
For $|t|>1$, we have
\begin{multline*}
\frac{d}{du}\log (au;t^{-1})_\infty = -\sum_{k=0}^\infty \frac{at^{-k}}{1-aut^{-k}}\\
= -\sum_{k=0}^\infty \sum_{r=0}^\infty a^{r+1}t^{-(r+1)k}u^r = -\sum_{r=1}^\infty \frac{a^r}{1-t^{-r}}u^{r-1}.
\end{multline*}
Using this observation, a direct computation reveals that
$$
\left(E^\natural_{n,m}\right)^\prime(u)/E^\natural_{n,m}(u) = \frac{d}{du}\log E^\natural_{n,m}(u) = -P^\natural_{n,m}(u).
$$
Comparing coefficients of $u^{r-1}$ in the equivalent power-series identity $\left(E^\natural_{n,m}\right)^\prime(u)=-P^\natural_{n,m}(u)E^\natural_{n,m}(u)$, we obtain the following analogues of Newton's formulae:
$$
re^\natural_r(x,y) = \sum_{s=1}^r (-1)^{s-1}p^\natural_s(x,y)e^\natural_{r-s}(x,y)\ \ \ (r\in\N^*).
$$
In particular, they make it possible to express $e_1^\natural(x,y),\ldots,e_{n+m}^\natural(x,y)$ in terms of $p_1^\natural(x,y),\ldots,p_{n+m}^\natural(x,y)$ and vice versa, which clearly entails that
$$
\mathbb{C}\left[e_1^\natural(x,y),\ldots,e_{n+m}^\natural(x,y)\right] = \mathbb{C}\left[p_1^\natural(x,y),\ldots,p_{n+m}^\natural(x,y)\right].
$$

Hence the assertion will follow once we prove that the deformed shifted power sums $p_r^\natural(x,y)$ with $r=1,\ldots,n+m$ are algebraically independent, which, in turn, is readily inferred from their Jacobian, see e.g.~Thm.~2.2 in \cite{ER93}.\footnote{We are grateful to Misha Feigin for explaining this to us.} Indeed, if we have a relation $F(p^\natural_1,\ldots,p^\natural_{n+m})\equiv 0$, the chain rule entails
\begin{equation*}
\begin{split}
\mathbf{0} &= \left(\frac{\partial F}{\partial x_1},\ldots,\frac{\partial F}{\partial x_n},\frac{\partial F}{\partial y_1},\ldots,\frac{\partial F}{\partial y_m}\right)\\
&= \left(\frac{\partial F}{\partial p^\natural_1},\ldots,\frac{\partial F}{\partial p^\natural_{n+m}}\right)\cdot \frac{\partial\left(p_1^\natural,\ldots,p_{n+m}^\natural\right)}{\partial(x_1,\ldots,x_n,y_1,\ldots,y_m)},
\end{split}
\end{equation*}
and, assuming $F$ is of minimal degree, $\partial F/\partial p^\natural_j\neq 0$ for some $j=1,\ldots,n+m$, so that the Jacobian (determinant) must be zero. However, by a direct computation, we see, in particular, that the coefficient of the monomial $x_2 x_3^2\cdots x_n^{n-1} y_1^n y_2^{n+1}\cdots y_m^{n+m-1}$ equals
$$
n!\prod_{i=1}^n t^{i(1-i)}\cdot \frac{(q^{n+1};q)_m}{(t^{-n-1};t^{-1})_m} (n+1)_m\prod_{j=1}^m q^{(n+j)(j-1)},
$$
which is manifestly non-zero.
\end{proof}

\subsection{Restriction interpretation}
\label{SubSec:Res}
Deviating slightly from the notation in Macdonald's book \cite{Mac95}, we write $\Lambda_N$, $N\in\N^*$, for the graded algebra of {\it complex} $S_N$-invariant polynomials $p(z)$ in $N$ variables $z=(z_1,\ldots,z_N)$. The inverse limit $\Lambda=\varprojlim_N\, \Lambda_N$ (in the category of graded algebras), i.e.~the algebra of symmetric functions, will play an important role in this section.

More specifically, we establish an interpretation of the deformed NS operators $\cH_{n,m}^r$\eqref{cHnmk} and MR operators $\cH_{n,m}^r$ \eqref{cDnmk} as restrictions of operators $\cH_\infty^{r}$ and $\cD_\infty^{r}$, respectively, on $\Lambda$, thereby generalising Thm.~5.4 in \cite{SV09}, which essentially amounts to the $r=1$ case, to all $r\in\N$.

To begin with, suppose that the parameters $q,t$ are non-special. Then, as recalled in Subsection \ref{SubSec:EigFuncs}, the super Macdonald polynomials $SP_\lambda(x,y)$ ($\lambda\in H_{n,m}$) span the algebra $\Lambda_{n,m;q,t}$. Hence, by Theorem \ref{Thm:HCiso}, each difference operator $\cH_{n,m}^f$ ($f\in\Lambda_{n,m;q,t}^\natural$) leaves $\Lambda_{n,m;q,t}$ invariant.

More generally, we can work directly with the symmetry conditions that characterise $\Lambda_{n,m;q,t}$ to establish the following result.

\begin{lemma}
\label{Lemma:pres}
Assume that $q,t\in\C^*$ are not roots of unity. Then, for each $f\in\Lambda_{n,m;q,t}^\natural$, the difference operator $\cH_{n,m}^f$ preserves the algebra $\Lambda_{n,m;q,t}$:
$$
\cH_{n,m}^f: \Lambda_{n,m;q,t}\to \Lambda_{n,m;q,t}.
$$
\end{lemma}

\begin{proof}
A proof of this lemma is provided in Appendix \ref{App:pres}.
\end{proof}

Now, with $z=(z_1,z_2,\ldots)$ and $w=(w_1,w_2,\ldots)$ two infinite sequences of variables, we recall from Eqs.~(2.5)-(2.6) and (4.13) in Chapter VI of \cite{Mac95} the kernel function
$$
\Pi(z;w;q,t) = \prod_{i,j}\frac{(tz_iw_j;q)_\infty}{(z_iw_j;q)_\infty}
$$
along with its expansions in terms of power sums and Macdonald symmetric functions:
\begin{equation}
\label{PiExps}
\Pi(z;w;q,t) = \sum_\lambda z_\lambda(q,t)^{-1} p_\lambda(z)p_\lambda(w) = \sum_\lambda P_\lambda(z;q,t)Q_\lambda(w;q,t)
\end{equation}
with
$$
z_\lambda(q,t) = \prod_{i\geq 1} i^{m_i} m_i!\cdot \prod_{i=1}^{l(\lambda)} \frac{1-q^{\lambda_i}}{1-t^{\lambda_i}},
$$
where $m_i=m_i(\lambda)$ denotes the number of parts of $\lambda$ equal to $i$.

Introducing the notation
$$
\cH_N(z;u)=\cH_{N,0}(z;u),\ \ \ \cD_N(z;u)=\cD_{N,0}(z;u),
$$
let us define operators $\cH_N^r$ ($r\in\N$) and $\cD_N^r$ ($r\in\N$) by the generating series expansions
$$
\cH_N(z;u) = \sum_{r=0}^\infty u^r\cH_N^r(z),\ \ \ \cD_N(z;u) = \sum_{r=0}^\infty (-u)^r\cD_N^r(z).
$$
One of their distinguishing features is stability under reductions of the number of variables $N$. More precisely, with the homomorphism
$$
\rho_{N,N-1}: \Lambda_N\to \Lambda_{N-1},\, p(z_1,\ldots,z_N)\mapsto p(z_1,\ldots,z_{N-1},0),
$$
the diagrams
\begin{equation*}
\begin{tikzcd}
\Lambda_N \arrow[r,"\rho_{N,N-1}"] \arrow[d,swap,"\cH_N^r"] &
\Lambda_{N-1} \arrow[d,"\cH_{N-1}^r"] \\
\Lambda_N \arrow[r,"\rho_{N,N-1}"] & \Lambda_{N-1}
\end{tikzcd}
\end{equation*}
and
\begin{equation*}
\begin{tikzcd}
\Lambda_N \arrow[r,"\rho_{N,N-1}"] \arrow[d,swap,"\cD_N^r"] &
\Lambda_{N-1} \arrow[d,"\cD_{N-1}^r"] \\
\Lambda_N \arrow[r,"\rho_{N,N-1}"] & \Lambda_{N-1}
\end{tikzcd}
\end{equation*}
are commutative for all $r\in\N$. To see this, it suffices to note that the eigenvalues of these operators are independent of $N$, cf.~\eqref{cGlamExpr2} and \eqref{EnatExpr} or see \cite{NS20} and \cite{Mac95}, respectively.

Hence, we have well-defined generating series
$$
\cH_\infty(z;u) := \sum_{r=0}^\infty u^r \cH_\infty^r(z),\ \ \ \cD_\infty(z;u) := \sum_{r=0}^\infty (-u)^r \cD_\infty^r(z),
$$
of operators
$$
\cH_\infty^r := \varprojlim_N\, \cH_N^r: \Lambda\to\Lambda,\ \ \ \cD_\infty^r := \varprojlim_N\, \cD_N^r: \Lambda\to\Lambda,
$$
which are simultaneously diagonalised by the Macdonald symmetric functions. Combining this fact with the latter expansion in \eqref{PiExps}, we arrive at the following lemma.

\begin{lemma}
\label{Lemma:PiKerIds}
For $|q|<1$, we have the kernel identities
\begin{equation*}
\cH_\infty(z;u)\Pi(z;w) = \cH_\infty(w;u)\Pi(z;w)
\end{equation*}
and
\begin{equation*}
\cD_\infty(z;u)\Pi(z;w) = \cD_\infty(w;u)\Pi(z;w).
\end{equation*}
\end{lemma}

From Thm.~5.8 in \cite{SV09}, we recall that the homomorphism
$$
\varphi_{n,m;q,t}: \Lambda\to \Lambda_{n,m;q,t},\ p_r(z_1,z_2,\ldots)\mapsto p_r((x_1,\ldots,x_n),(y_1,\ldots,y_m);q,t),
$$
with the deformed Newton sums
$$
p_r(x,y;q,t) = \sum_{i=1}^n x_i^r + \frac{1-q^r}{1-t^{-r}}\sum_{j=1}^m y_j^r\ \ \ (r\in\N^*),
$$
is surjective whenever $q,t$ are non-special. (Note that $t$ in {\it loc.~cit.} corresponds to $t^{-1}$ here.)
Assuming $|q|<1$, we also recall that
\begin{equation}
\label{varphiPi}
\varphi_{n,m;q,t}(\Pi(z;w)) = \Pi_{n,m;\infty}(x,y;w)
\end{equation}
with
\begin{equation}
\label{PiDef}
\Pi_{n,m;\infty}(x,y;w) = \prod_{i=1}^n\prod_{k=1}^\infty \frac{(tx_iw_k;q)_\infty}{(x_iw_k;q)_\infty}\cdot \prod_{j=1}^m\prod_{k=1}^\infty (1-ty_jw_k),
\end{equation}
see Property (ii) in Lemma 5.5 in \cite{SV09}.

Setting $w_k=0$ for $k>N$, with $N\in\N^*$, in \eqref{PiDef}, we obtain the meromorphic function
\begin{equation}
\begin{split}
\Pi_{n,m;N}(x,y;w) &:= \prod_{i=1}^n\prod_{k=1}^N \frac{(tx_iw_k;q)_\infty}{(x_iw_k;q)_\infty}\cdot \prod_{j=1}^m\prod_{k=1}^N (1-ty_jw_k)\\
&= \Phi_{n,m;N,0}(tx,ty;w),
\end{split}
\end{equation}
cf.~\eqref{kerFunc}. Hence, since $\cH_{n,m}(x,y;u)$ and $\cD_{n,m}(x,y;u)$ are invariant under $x\to tx$ and $x\to ty$, we may substitute $\Pi_{n,m;N}$ for $\Phi_{n,m;N,0}$ in the $M=0$ instance of Theorem \ref{Thm:kerIds}. By stability under reductions of the number of variables $N$, the resulting (sequence of) kernel identities amount to
\begin{equation}
\label{cHinfKerId}
\cH_{n,m}(x,y;u)\Pi_{n,m;\infty}(x,y;w) = \cH_\infty(z;u)\Pi_{n,m;\infty}(x,y;w)
\end{equation}
and
\begin{equation}
\label{cDinfKerId}
\cD_{n,m}(x,y;u)\Pi_{n,m;\infty}(x,y;w) = \cD_\infty(z;u)\Pi_{n,m;\infty}(x,y;w).
\end{equation}

Using Lemma \ref{Lemma:PiKerIds} as well as \eqref{varphiPi} and \eqref{cHinfKerId}, we deduce
\begin{equation*}
\begin{split}
\varphi_{n,m;q,t} \big(\cH_\infty(z;u)\Pi(z;w)\big) &= \varphi_{n,m;q,t} \big(\cH_\infty(w;u)\Pi(z;w)\big)\\
&= \cH_\infty(w;u)\big(\varphi_{n,m;q,t} \Pi(z;w)\big)\\
&= \cH_{n,m}(x,y;u)\big(\varphi_{n,m;q,t} \Pi(z;w)\big).
\end{split}
\end{equation*}
Substituting the former expansion in \eqref{PiExps} and comparing coefficients, we find that
\begin{equation}
\label{interRelcH}
\big(\varphi_{n,m;q,t}\circ \cH_\infty(u)\big) p_\lambda = \big(\cH_{n,m}(u)\circ\varphi_{n,m;q,t}\big) p_\lambda
\end{equation}
for all partitions $\lambda$. Using \eqref{cDinfKerId} instead of \eqref{cHinfKerId}, we obtain \eqref{interRelcH} with $\cH_\infty(u)\to \cD_\infty(u)$ and $\cH_{n,m}(u)\to \cD_{n,m}(u)$.

Since the $p_\lambda$ span $\Lambda$, the above arguments and Lemma \ref{Lemma:pres} yield the following result.

\begin{theorem}
\label{Thm:res}
For all $q,t\in\C^*$ that are not roots of unity, the diagrams
\begin{equation*}
\begin{tikzcd}
\Lambda \arrow[r,"\varphi_{n,m;q,t}"] \arrow[d,swap,"\cH_\infty^{r}"] &
\Lambda_{n,m;q,t} \arrow[d,"\cH_{n,m}^{r}"] \\
\Lambda \arrow[r,"\varphi_{n,m;q,t}"] & \Lambda_{n,m;q,t}
\end{tikzcd}
\end{equation*}
and
\begin{equation*}
\begin{tikzcd}
\Lambda \arrow[r,"\varphi_{n,m;q,t}"] \arrow[d,swap,"\cD_\infty^{r}"] &
\Lambda_{n,m;q,t} \arrow[d,"\cD_{n,m}^{r}"] \\
\Lambda \arrow[r,"\varphi_{n,m;q,t}"] & \Lambda_{n,m;q,t}
\end{tikzcd}
\end{equation*}
are commutative for all $r\in\N$.
\end{theorem}

\begin{remark}
This result has an interesting (algebro-)geometric interpretation. Assuming $q^i/t^j\neq 1$ for all $i=1,\ldots,n$ and $j=1,\ldots,m$, it was proved in Thm.~5.1 in \cite{SV09} that the algebra $\Lambda_{n,m;q,t}$ is finitely generated, so that a corresponding (affine) variety
$$
\Delta_{n,m;q,t} := \mathrm{Spec}\, \Lambda_{n,m;q,t}
$$
could be introduced. Restricting attention further to non-special parameter values $q,t$, the homomorphism $\varphi_{n,m;q,t}:\Lambda\to\Lambda_{n,m;q,t}$ is surjective, and thus yields an embedding $\phi: \Delta_{n,m;q,t}\to\cM$, with $\cM:=\mathrm{Spec}\, \Lambda$ called the (infinite-dimensional) Macdonald variety.

Hence, for $q,t$ non-special, the deformed NS operators $\cH_{n,m}^{r}$ can be viewed as the restrictions of the operators $\cH_\infty^{r}$ onto the subvariety $\Delta_{n,m;q,t}\subset\cM$, and similarly for the deformed MR operators $\cD_{n,m}^{r}$.
\end{remark}

\begin{appendix}

\section{Proofs of Lemmas}
\label{App:Lemmas}
In this Appendix, we provide proofs of the Lemmas in the main text.

\subsection{Commutativity}

\subsubsection{Proof of Lemma \ref{Lemma:diffop}}
\label{App:diffop}
From \eqref{kerFunc}, it is readily inferred that the equality $L_{n,m}(x,y)\Phi_{n,m;N,0}(x,y;z)=0$ is equivalent to
$$
\sum_{\substack{\mu\in\N^n, \nu\in\N^m\\ |\mu|+|\nu|\leq d}}a_{\mu,\nu}(x,y) \prod_{i=1}^n\prod_{j=1}^N \frac{(t^{-1}x_iz_j;q)_{\mu_i}}{(x_iz_j;q)_{\mu_i}}\cdot \prod_{i=1}^m\prod_{j=1}^N \frac{1-t^{-\nu_i}y_iz_j}{1-y_iz_j} = 0.
$$
We write this formula as
\begin{equation}
\label{avarphisum}
\sum_{\substack{\mu\in\N^n, \nu\in\N^m\\ |\mu|+|\nu|\leq d}}a_{\mu,\nu}(x,y)\varphi_{\mu,\nu}(x,y;z) = 0,
\end{equation}
where
$$
\varphi_{\mu,\nu}(x,y;z) = \prod_{i=1}^n\prod_{j=1}^N \frac{(t^{-1}x_iz_j;q)_{\mu_i}}{(x_iz_j;q)_{\mu_i}}\cdot \prod_{i=1}^m\prod_{j=1}^N \frac{1-t^{-\nu_i}y_iz_j}{1-y_iz_j}.
$$
For $(\alpha,\beta)\in\N^n\times\N^m$ such that $|\alpha|+|\beta|\leq d$, we take $N=n+dm-|\beta|$ and specialise the $z$-variables to
\begin{multline*}
z_{\alpha,\beta} = (t/q^{\alpha_1}x_1,t/q^{\alpha_2}x_2,\ldots,t/q^{\alpha_n}x_n;
t^{\beta_1+1}/y_1,t^{\beta_1+2}/y_1,\ldots,t^d/y_1;\\
\ldots;t^{\beta_m+1}/y_m,t^{\beta_m+2}/y_m,\ldots,t^d/y_m).
\end{multline*}
Then we get
\begin{equation*}
\begin{split}
\varphi_{\mu,\nu}(x,y;z_{\alpha,\beta}) &= \prod_{i,j=1}^n \frac{(x_i/q^{\alpha_j}x_j;q)_{\mu_i}}{(tx_i/q^{\alpha_j}x_j;q)_{\mu_i}}\cdot \prod_{i=1}^n\prod_{j=1}^m \frac{(t^{\beta_j}x_i/y_j;q,t)_{\mu_i,d-\beta_j}}{(t^{\beta_j+1}x_i/y_j;q,t)_{\mu_i,d-\beta_j}}\\
&\quad \cdot \prod_{i=1}^m\prod_{j=1}^n \frac{1-t^{1-\nu_i}y_i/q^{\alpha_j}x_j}{1-ty_i/q^{\alpha_j}x_j}\cdot \prod_{i,j=1}^m \frac{(t^{\beta_j+1-\nu_i}y_i/y_j;t)_{d-\beta_j}}{(t^{\beta_j+1}y_i/y_j;t)_{d-\beta_j}},
\end{split}
\end{equation*}
where
$$
(a;q,t)_{k,l} = \prod_{i=0}^{k-1}\prod_{j=0}^{l-1} (1-aq^it^j).
$$
Since $\varphi_{\mu,\nu}(x,y;z_{\alpha,\beta})$ contains the ``diagonal'' factors
$$
\prod_{i=1}^n \frac{(q^{-\alpha_i};q)_{\mu_i}}{(tq^{-\alpha_i};q)_{\mu_i}}\cdot \prod_{j=1}^m \frac{(t^{\beta_j+1-\nu_j};t)_{d-\beta_j}}{(t^{\beta_j+1};t)_{d-\beta_j}},
$$
it vanishes if $\mu_i>\alpha_i$ for some $i\in\{1,\ldots,n\}$ or if $\nu_j>\beta_j$ for some $j\in\{1,\ldots,m\}$, which clearly entails that the $z\to z_{\alpha,\beta}$ specialisation of \eqref{avarphisum} is given by
\begin{equation}
\label{avarphisum2}
a_{\alpha,\beta}(x,y)\varphi_{\alpha,\beta}(x,y;z_{\alpha,\beta}) + \sum_{\substack{\mu\in\N^n, \nu\in\N^m\\ |\mu|+|\nu|<|\alpha|+|\beta|}}a_{\mu,\nu}(x,y)\varphi_{\mu,\nu}(x,y;z_{\alpha,\beta}) = 0.
\end{equation}

Now assume that $L_{n,m}(x,y)$ is a non-zero difference operator. Then we let $d\in\N$ be the smallest non-negative integer such that $a_{\alpha,\beta}\neq 0$ for some $(\alpha,\beta)\in\N^n\times\N^m$ with $|\alpha|+|\beta|=d$. By \eqref{avarphisum2}, we have $a_{\alpha,\beta}(x,y)\varphi_{\alpha,\beta}(x,y;z_{\alpha,\beta})=0$ and, since the meromorphic function $\varphi_{\alpha,\beta}(x,y;z_{\alpha,\beta})$ is non-zero, it follows that $a_{\alpha,\beta}=0$. Hence we have reached a contradiction and the lemma follows.

\subsection{Harish-Chandra isomorphism}

\subsubsection{Proof of Lemma \ref{Lemma:gknat}}
\label{App:gknat}
We begin by recalling a few definitions and results from the literature that we make use of in the proof.

We write $\Lambda_{N,t}$ for the algebra of complex polynomials $p\in\C[z_1,\ldots,z_N]$ that are symmetric in the ``shifted'' variables $z_1,z_2t^{-1},\ldots,z_Nt^{N-1}$, and note that it is filtered by the degree of the polynomials:
$$
\Lambda_{N,t}^{\leq 0}\subset \Lambda_{N,t}^{\leq 1}\subset\cdots \subset \Lambda_{N,t}^{\leq r}\subset\cdots\ \ \ (r\in\N),
$$
with $\Lambda_{N,t}^{\leq r}$ the subspace of such polynomials of degree at most $r$. The inverse limit $\Lambda_t$ of the filtered algebras $\Lambda_{N,t}$ with respect to the homomorphisms
$$
\Lambda_{N,t}\to \Lambda_{N-1,t},\ \ \ p(z_1,\ldots,z_{N-1},z_N)\mapsto p(z_1,\ldots,z_{N-1},1)\ \ \ (N\in\N^*),
$$
is the so-called algebra of shifted symmetric functions \cite{Oko98}. It is, for example, generated by the shifted power sums
$$
p_r^*(z;t) = \sum_{i\geq 1}(z_i^r-1)t^{r(1-i)}\ \ \ (r\in\N^*).
$$

From Thm.~6.2 in \cite{SV09}, we recall that the algebra homomorphism
\begin{equation}
\label{varphinat}
\varphi_{n,m;q,t}^\natural: \Lambda_t\to \Lambda_{n,m;q,t}^\natural,\ p_r^*(z;t)\mapsto p_r^\natural(x;y;q,t),
\end{equation}
with the deformed shifted power sums
\begin{equation}
\label{pknat}
p_r^\natural(x;y;q,t) = \sum_{i=1}^n (x_i^r-1)t^{r(1-i)}+\frac{1-q^r}{1-t^{-r}}\sum_{j=1}^m (y_j^r-t^{-rn})q^{r(j-1)}\ \ \ (r\in\N^*),
\end{equation}
is surjective under the assumption that $q,t$ are non-special.

Our proof strategy is to exhibit shifted symmetric functions $g^*_r(z)=g^*_r(z;q,t)$ ($r\in\N^*$) that freely generate $\Lambda_t$ and are such that $\varphi_{n,m;q,t}(g^*_r(z))=g^\natural_r(x,y)$, cf.~\eqref{gknat}.

More specifically, let us consider the generating function expansion
\begin{equation}
\label{gkstar}
G^*_N(z;u) := \prod_{i=1}^N \frac{(z_it^{2-i}u;q)_\infty}{(t^{2-i}u;q)_\infty}\frac{(t^{1-i}u;q)_\infty}{(z_it^{1-i}u;q)_\infty} = \sum_{r\geq 0}g_r^*(z_1,\ldots,z_N)u^r.
\end{equation}
Since the product is manifestly symmetric in the variables $z_it^{1-i}$, it is clear that $g_r^*(z_1,\ldots,z_N)\in\Lambda_{N,t}^r$. Moreover, setting $z_N=1$ in \eqref{gkstar}, we find the following stability properties:
$$
G^*_N(z_1,\ldots,z_{N-1},1;u) = G^*_{N-1}(z_1,\ldots,z_{N-1};u),
$$
$$
g_r^*(z_1,\ldots,z_{N-1},1) = g_r^*(z_1,\ldots,z_{N-1}),
$$
so that we can define shifted symmetric functions $g_r^*(z)$ as the coefficients of $u^r$ in the power series expansion of the infinite product
$$
G^*(z;u) := \prod_{i\geq 1} \frac{(z_it^{2-i}u;q)_\infty}{(t^{2-i}u;q)_\infty}\frac{(t^{1-i}u;q)_\infty}{(z_it^{1-i}u;q)_\infty}.
$$

Fixing $r\in\N$, we claim that the
$$
g_\lambda^*(z) := \prod_{i\geq 1} g^*_{\lambda_i}(z),\ \ \ |\lambda|\leq r,
$$
constitute a basis in $\Lambda^{\leq r}_t$. To prove this claim, we may and shall work in $\Lambda^{\leq r}_{N,t}$ as long as $N\geq r$, since the corresponding projection $\rho^*_N: \Lambda_t^{\leq r}\to \Lambda_{N,t}^{\leq N}$ is an isomorphism. We observe that
$$
g_r^*(z_1,\ldots,z_N) = g_r(z_1,z_2t^{-1},\ldots,z_Nt^{1-N})+\text{lower degree terms},
$$
where $g_r(x_1,\ldots,x_N)=g_r(x_1,\ldots,x_N;q,t)$ are the symmetric polynomials from Eq.~(2.8) in Section VI.2 of \cite{Mac95}:
$$
\prod_{i=1}^N \frac{(tx_iu;q)_\infty}{(x_iu;q)_\infty} = \sum_{r\geq 0}g_r(x_1,\ldots,x_N)u^r.
$$
Recalling from (2.19) in {\it loc.~cit.}~that the
$$
g_\lambda(x_1,\ldots,x_N) := \prod_{i\geq 1}g_{\lambda_i}(x_1,\ldots,x_N),\ \ \ |\lambda|=r,
$$
form a basis in $\Lambda_N^r$, the claim is readily established by induction in $r$. In other words, we have just shown that the $g_r^*(z)$ $(r\in\N^*)$ freely generate the algebra of shifted symmetric functions $\Lambda_t$.

Next, we claim that
\begin{equation}
\label{gkmap}
\varphi_{n,m;q,t}^\natural(g_r^*(z)) = g_r^\natural(x,y)\ \ \ (r\in\N^*).
\end{equation}
Taking this claim for granted, it becomes clear that the polynomials $g_r^\natural(x,y)$ ($r\in\N^*$) generate $\Lambda^\natural_{n,m;q,t}$, since $\varphi_{n,m;q,t}^\natural$ \eqref{varphinat} is surjective and the shifted symmetric functions $g_r^*$ ($r\in\N^*$) generate $\Lambda_t$.

To complete the proof, there remains only to verify \eqref{gkmap}. Since we have no explicit formulae for $g_r^*$ and $g_r^\natural$ in terms of $p_r^*$ and $p_r^\natural$, respectively, we shall work on the level of generating functions. First, we deduce the equalities
\begin{equation}
\label{Gstcmp}
\begin{split}
G^*(z;u) &= \prod_{i\geq 1}\prod_{a\geq 0}\frac{1-q^az_it^{2-i}u}{1-q^at^{2-i}u} \frac{1-q^at^{1-i}u}{1-q^az_it^{1-i}u}\\
&= \exp\Bigg(\sum_{i\geq 1}\sum_{a\geq 0}\big[\log(1-q^az_it^{2-i}u)-\log(1-q^at^{2-i}u)\\
&\qquad\qquad\qquad\qquad +\log(1-q^at^{1-i}u)-\log(1-q^az_it^{1-i}u)\big]\Bigg)\\
&= \exp\Bigg(\sum_{i\geq 1}\sum_{a\geq 0}\sum_{r\geq 1}\frac{u^rq^{ar}}{r}(1-t^r)(z_i^r-1)t^{r(1-i)}\Bigg)\\
&= \exp\Bigg(\sum_{r\geq 1}\frac{u^r}{r}\frac{1-t^r}{1-q^r}p_r^*(z;t)\Bigg).
\end{split}
\end{equation}
From \eqref{varphinat}, it follows that
$$
\varphi^\natural_{n,m;q,t}(G^*(z;u)) = \exp\Bigg(\sum_{r\geq 1}\frac{u^r}{r}\frac{1-t^r}{1-q^r}p_r^\natural(x,y)\Bigg).
$$
Substituting \eqref{pknat}, reversing the steps in \eqref{Gstcmp} and comparing the end result with \eqref{Gnat}, we obtain
$$
\varphi^\natural_{n,m;q,t}(G^*(z;u)) = G^\natural_{n,m}(x,y;u),
$$
which clearly is equivalent to \eqref{gkmap}. This completes the proof of Lemma \ref{Lemma:gknat}.

\subsection{Restriction interpretation}
\subsubsection{Proof of Lemma \ref{Lemma:pres}}
\label{App:pres}
By the definition of $\cR_{n,m;q,t}$ \eqref{cRnm}, it suffices to prove the claim for the deformed NS operators $\cH_{n,m}^{r}$ ($r\in\N^*$) and, for convenience, we shall work with their generating function $\cH_{n,m}(u)$ \eqref{cHnm}.

Observing that their coefficient functions $B_{\mu,I}$ \eqref{BmuI} satisfy
$$
B_{\mu,I}(\sigma x,\tau y) = B_{\sigma^{-1}\mu,\tau^{-1}I}(x,y),\ \ \ (\sigma,\tau)\in S_n\times S_m,
$$
it becomes clear that $\cH_{n,m}(x,y;u)$ commutes with the action of $S_n\times S_m$:
\begin{equation}
\label{inv}
(\sigma,\tau)\circ \cH_{n,m}(x,y;u) = \cH_{n,m}(x,y;u)\circ (\sigma,\tau)
\end{equation}
for all $(\sigma,\tau)\in S_n\times S_m$.

Given any $f\in\Lambda_{n,m;q,t}$, we let
\begin{equation*}
\begin{split}
F &= \cH_{n,m}(u) f\\
&= \frac{(t^{1-n}q^mu;q)_\infty}{(tu;q)_\infty}\sum_{\mu\in\N^n}\sum_{I\subseteq\{1,\ldots,m\}} (t^{1-n}q^mu)^{|\mu|}(-tu)^{|I|}q^{\binom{|I|}{2}}B_{\mu,I} T_{q,x}^\mu T_{t,y}^{-I}f.
\end{split}
\end{equation*}
From \eqref{BmuI}, it is clear that $F$ is a rational function in $x$ and $y$ whose poles are at most simple and located only along
\begin{equation}
\label{xpls}
x_j = q^{k_i}x_i,\ \ \ k_i=0,1,\ldots,\mu_i,\ \ 1\leq i\neq j\leq n,
\end{equation}
\begin{equation}
\label{ypls}
y_i = y_j,\ \ \ 1\leq i\neq j\leq m,
\end{equation}
or
\begin{equation}
\label{xxipls}
y_j = q^{\mu_i}x_i,\ \ y_j = q^{\mu_i-1}x_i,\ \ \ 1\leq i\leq n,\ \ 1\leq j\leq m.
\end{equation}
Moreover, thanks to the invariance property \eqref{inv}, $F$ is symmetric in both $x$ and $y$. Hence it will follow that $F$ is a polynomial once we can show that it has no poles of the form \eqref{xpls}--\eqref{ypls} with $j=1$ and $i=2$ or \eqref{xxipls} with $j=i=1$.

First, we consider the $y$-independent poles \eqref{xpls} with $j=1$ and $i=2$. Let us fix $\mu\in\N^n$ such that at least one of the elements $\mu_1$, $\mu_2$ is non-zero, take $0\leq k_2\leq\mu_2$, and let
$$
\tilde{\mu} = \sigma_{12}\mu-k_2(e_1-e_2) = (\mu_2-k_2,\mu_1+k_2,\mu_3,\ldots,\mu_n),
$$
where $\sigma_{12}$ denotes the transposition that acts on $\mu=(\mu_1,\ldots,\mu_n)$ by interchanging $\mu_1$ and $\mu_2$. We note that both $B_{\mu,I}$ and $B_{\tilde{\mu},I}$ have a simple pole along $x_1=q^{k_2}x_2$. (In the excluded case $\mu_1=\mu_2=0$ these operator coefficients, which then coincide, are regular along $x_1=x_2$.) In order to compare the corresponding residues, we observe that
$$
\Delta(q^{\tilde{\mu}}x) = -\Delta(q^\mu x),\ \ \ x_1 = q^{k_2}x_2,
$$
whenever $k_2>0$; and that, when interpreted in terms of residues along $x_1-x_2$, the equality holds true also for $k_2=0$. Combining this equality with the identity
$$
(aq^{k_2};q)_{\mu_1}(a;q)_{\mu_2} = (aq^{k_2};q)_{\mu_2-k_2}(a;q)_{\mu_1+k_2},\ \ \ a\in\C,
$$
it is readily verified that
$$
(x_1-q^{k_2}x_2)B_{\tilde{\mu},I}(x,y) = -(x_1-q^{k_2}x_2)B_{\mu,I}(x,y),\ \ \ x_1 = q^{k_2}x_2.
$$
It follows that the residues along $x_1 = q^{k_2}x_2$ of the two terms in
$$
B_{\mu,I}T_{q,x}^\mu T_{t,y}^{-I}f+B_{\tilde{\mu},I}T_{q,x}^{\tilde{\mu}} T_{t,y}^{-I}f
$$
cancel, and consequently that $F$ has no poles of the form \eqref{xpls}.

Second, we focus attention on the $x$-independent pole \eqref{ypls} with $j=1$ and $i=2$. We note that $B_{\mu,I}$ is regular along $y_1=y_2$ unless $1\in I$ and $2\notin I$ or vice versa. For such an index set $I$, it is clear from \eqref{BmuI} that
$$
(y_1-y_2)B_{\mu,\sigma_{12}I}(x,y) = -(y_1-y_2)B_{\mu,I}(x,y),\ \ \ y_1 = y_2,
$$
so that the residues along $y_1=y_2$ of the terms in
$$
B_{\mu,I}T_{q,x}^\mu T_{t,y}^{-I}f+B_{\mu,\sigma_{12}I}T_{q,x}^{\mu} T_{t,y}^{-\sigma_{12}I}f
$$
cancel and $F$ is thus free of poles of the form \eqref{ypls}.

Third, we handle the poles \eqref{xxipls} with $i=j=1$. Taking $I\subset\{1,\ldots,m\}$ such that $1\notin I$, we deduce
\begin{equation*}
\begin{split}
\frac{B_{\mu+e_1,I}(x,y)}{B_{\mu,I\cup\{1\}}(x,y)} &= q^{n-1}\prod_{i=2}^n \left(\frac{y_1-q^{\mu_i}x_i}{qy_1-q^{\mu_i}x_i} \cdot \frac{q^{\mu_1+1}x_1-q^{\mu_i}x_i}{q^{\mu_1}x_1-q^{\mu_i}x_i}\right)\\
&\quad \cdot (t/q)^{n-1} \prod_{i=2}^n \left(\frac{qy_1-x_i}{ty_1-x_i} \cdot \frac{tq^{\mu_1}x_1-x_i}{q^{\mu_1+1}x_1-x_i}\right)\\
&\quad \cdot (t/q)\cdot \frac{qy_1-x_1}{ty_1-x_1} \cdot \frac{1-tq^{\mu_1}}{1-q^{\mu_1+1}}\\
&\quad \cdot \prod_{i\in I} \left(\frac{qy_1-y_i}{y_1-y_i} \cdot \frac{q^{\mu_1}x_1-y_i}{q^{\mu_1+1}x_1-y_i}\right)\\
&\quad \cdot q^{|I|+1-m} \prod_{i\notin I\cup\{1\}}\left(\frac{y_1-y_i}{y_1-qy_i} \cdot \frac{q^{\mu_1}x_1-qy_i}{q^{\mu_1}x_1-y_i}\right),
\end{split}
\end{equation*}
from which it clearly follows that $B_{\mu+e_1,I}(x,y)/B_{\mu,I\cup\{1\}}(x,y)=t^nq^{|I|-m}$ for $y_1=q^{\mu_1}x_1$. Since $\binom{|I|+1}{2}-\binom{|I|}{2}=|I|$, we can thus conclude that
\begin{multline*}
(t^{1-n}q^mu)^{|\mu|+1}(-tu)^{|I|}q^{\binom{|I|}{2}}(y_1-q^{\mu_i}x_i)B_{\mu+e_1,I}(x,y)\\
+(t^{1-n}q^mu)^{|\mu|}(-tu)^{|I|+1}q^{\binom{|I|+1}{2}}(y_1-q^{\mu_i}x_i)B_{\mu,I\cup\{1\}}(x,y) = 0,\ \ \ y_1=q^{\mu_1}x_1,
\end{multline*}
which, together with the symmetry condition \eqref{qinv} with $i=j=1$, entails residue cancelation along $y_1=q^{\mu_1}x_1$ in
\begin{multline*}
(t^{1-n}q^mu)^{|\mu|+1}(-tu)^{|I|}q^{\binom{|I|}{2}}B_{\mu+e_1,I}T_{q,x}^{\mu+e_1} T_{t,y}^{-I}f\\
+(t^{1-n}q^mu)^{|\mu|}(-tu)^{|I|+1}q^{\binom{|I|+1}{2}}B_{\mu,I\cup\{1\}}T_{q,x}^\mu T_{t,y}^{-I\cup\{1\}}f,
\end{multline*}
and so $F$ is regular also along the hyperplanes \eqref{xxipls}.

In summary, we have shown that $F$ is a $S_n\times S_m$-invariant polynomial in $x$ and $y$, and there remains only to verify the symmetry conditions \eqref{qinv}. Restricting attention to the hyperplane $x_1=y_1$, we have the following implications:
\begin{enumerate}
\item[(I)] $1\in I\Rightarrow T_{t,y_1}^{-1}B_{\mu,I} = 0,$
\item[(II)] $\mu_1\geq 1\land 1\notin I\Rightarrow T_{q,x_1}B_{\mu,I} = 0,$
\item[(III)] $\mu_1=0\land 1\notin I\Rightarrow \big(T_{q,x_1}-T_{t,y_1}^{-1}\big)B_{\mu,I} = 0,$
\item[(IV)] $\mu_1\geq 1\land 1\notin I\Rightarrow T_{t,y_1}^{-1}B_{\mu,I}+t^nq^{|I|-m}T_{q,x_1}B_{\mu-e_1,I\cup\{1\}} = 0.$
\end{enumerate}
Implication (I) is due to the factor $1-x_1/ty_1$, present in the next to last product in $B_{\mu,I}$; Implication (II) is due to $1-x_1/qy_1$, contained in the last product; and Implications (III)--(IV) are straightforward to verify by direct (albeit somewhat lengthy) computations when isolating the $\mu_1$- and $y_1$-dependent factors in $B_{\mu,I}$ and $B_{\mu-e_1,I\cup\{1\}}$.

By Implication (III), we have
\begin{multline*}
\big(T_{q,x_1}-T_{t,y_1}^{-1}\big)F
= \sum_{\substack{\mu\in\N^n\\ \mu_1\geq 1}}\sum_{\substack{I\subset\{1,\ldots,m\}\\ 1\notin I}}(t^{1-n}q^mu)^{|\mu|}(-tu)^{|I|}q^{\binom{|I|}{2}}\\
\cdot \big(T_{q,x_1}-T_{t,y_1}^{-1}\big)\big(B_{\mu,I}T_{q,x}^\mu T_{t,y}^{-I}f-t^nq^{|I|-m}B_{\mu-e_1,I\cup\{1\}}T_{q,x}^{\mu-e_1}T_{t,y}^{-I\cup\{1\}}f\big)
\end{multline*}
when $x_1=y_1$. Using Implications (I)--(II), we deduce
\begin{multline*}
\big(T_{q,x_1}-T_{t,y_1}^{-1}\big)\big(B_{\mu,I}T_{q,x}^\mu T_{t,y}^{-I}f-t^nq^{|I|-m}B_{\mu-e_1,I\cup\{1\}}T_{q,x}^{\mu-e_1}T_{t,y}^{-I\cup\{1\}}f\big)\\
= -\big(T_{t,y_1}^{-1}B_{\mu,I}+t^nq^{|I|-m}T_{q,x_1}B_{\mu-e_1,I\cup\{1\}}\big)T_{q,x}^\mu T_{t,y}^{-I\cup\{1\}}f,
\end{multline*}
which, by Implication (IV), vanishes along $x_1=y_1$. This concludes our verification of the symmetry conditions \eqref{qinv}, and the Lemma follows.

\end{appendix}

\section*{Acknowledgments}
We would like to thank Farrokh Atai and Misha Feigin for valuable discussions.
M.N.~is grateful to the Knut and Alice Wallenberg Foundation for funding his guest professorship at KTH.
Financial support from the Swedish Research Council is acknowledged by M.H.~(Project-id~2018-04291) and H.R.~(Project-id~2020-04221).


\begin{thebibliography}{HLNR21}

\bibitem[AHL14]{AHL14}
Atai, F., Halln\"as, M., Langmann, E.:
Source Identities and Kernel Functions for Deformed (Quantum) Ruijsenaars Models. 
{\em Lett.~Math.~Phys.}~104, 811--835 (2014).

\bibitem[AL17]{AL17}
Atai, F., Langmann, E.:
Deformed Calogero-Sutherland model and fractional quantum Hall effect. 
{\em J.~Math.~Phys.}~58, 011902, 27 pp.~(2017).

\bibitem[Cha00]{Cha00}
Chalykh, O.A.:
Bispectrality for the quantum Ruijsenaars model and its integrable deformation.
{\em J.~Math.~Phys.}~41, 5139--5167 (2000).

\bibitem[Cha02]{Cha02}
Chalykh, O.A.:
Macdonald polynomials and algebraic integrability.
{\em Adv.~Math.}~166, 193--259 (2002).

\bibitem[CFV98]{CFV98}
Chalykh, O.A., Feigin, M.V., Veselov, A.P.:
New integrable generalizations of Calogero-Moser quantum problem.
{\em J.~Math.~Phys.}~39, 695--703 (1998).

\bibitem[DL15]{DL15}
Desrosiers, P, Liu, D.-Z.:
Selberg integrals, super-hypergeometric functions and applications to $\beta$-ensembles of random matrices.
{\em Random Matrices Theory Appl.}~04, 1550007, 59 pp.~(2015).

\bibitem[ER93]{ER93}
Ehrenborg, R., Rota, G.-C.:
Apolarity and canonical forms for homogenous polynomials.
{\em Europ.~J.~Combinatorics} 14, 157--181 (1993).

\bibitem[FS14]{FS14}
Feigin, M., Silantyev, A.:
Generalized Macdonald--Ruijsenaars systems.
{\em Adv.~Math.}~205, 144--192 (2014).

\bibitem[GS04]{GS04}
Gasper, G., Rahman, M.:
{\em Basic hypergeometric series}, second edition.
Encyclopedia od Mathematics and its Applications 96, Cambridge University Press, Cambridge (2004).

\bibitem[HLNR21]{HLNR21}
Halln\"as, M, Langmann, E., Noumi, M., Rosengren, H.:
Higher order deformed elliptic Ruijsenaars operators,
in preparation (2021).

\bibitem[Hei46]{Hei46}
Heine, E.:
\"Uber die Reihe $1+\frac{(q^\alpha-1)(q^\beta-1)}{(q-1)(q^\gamma-1)}x+\frac{(q^\alpha-1)(q^{\alpha+1}-1)(q^\beta-1)(q^{\beta+1}-1)}{(q-1)(q^2-1)(q^\gamma-1)(q^{\gamma+1}-1)}x^2+\cdots$.
{\em J.~Reine Angew.~Math.}~32, 210--312 (1846).

\bibitem[Hei47]{Hei47}
Heine, E.:
Untersuchungen \"uber die Reihe $1+\frac{(q^\alpha-1)(q^\beta-1)}{(q-1)(q^\gamma-1)}x+\frac{(q^\alpha-1)(q^{\alpha+1}-1)(q^\beta-1)(q^{\beta+1}-1)}{(q-1)(q^2-1)(q^\gamma-1)(q^{\gamma+1}-1)}x^2+\cdots$.
{\em J.~Reine Angew.~Math.}~34, 285--328 (1847).

\bibitem[Kaj04]{Kaj04}
Kajihara, Y.:
Euler transformation formula for multiple basic hypergeometric series of type $A$ and some applications.
{\em Adv.~Math.}~187, 53--97 (2004).

\bibitem[KN03]{KN03}
Kajihara, Y., Noumi, M.:
Multiple elliptic hypergeometric series. An approach from the Cauchy determinant.
{\em Indag.~Math.}~ (N.S.)~14, 395--421 (2003).

\bibitem[Mac95]{Mac95} I.G.~Macdonald,
{\em Symmetric functions and Hall polynomials}, second edition.
Oxford University Press, New York (1995).

\bibitem[NS20]{NS20}
Noumi, M., Sano, A.:
An infinite family of higher-order difference operators that commute with Ruijsenaars operators of type A.
arXiv:2012.03135 (2020).

\bibitem[Oko98]{Oko98}
Okounkov, A.:
(Shifted) Macdonald polynomials: q-integral representation and combinatorial formula.
{\em Compositio Math.}~112, 147--182 (1998).

\bibitem[Rui87]{Rui87}
Ruijsenaars, S.N.M.:
Complete integrability of relativistic Calogero-Moser systems and elliptic function identities. 
{\em  Commun.~Math.~Phys.}~110, 191--213 (1987).

\bibitem[Ser02]{Ser02}
Sergeev, A.N.:
The Calogero operator and Lie superalgebras.
{\em Theoret. and Math. Phys.}~131, 747--764 (2002).

\bibitem[SV04]{SV04}
Sergeev, A.N., Veselov, A.P.:
Deformed quantum Calogero-Moser problems and Lie superalgebras.
{\em Commun.~Math.~Phys.}~245, 249--278 (2004).

\bibitem[SV05]{SV05}
Sergeev, A.N., Veselov, A.P.:
Generalised discriminants, deformed Calogero-Moser-Sutherland operators and super-Jack polynomials.
{\em Adv.~Math.}~192, 341--375 (2005).

\bibitem[SV09]{SV09}
Sergeev, A.N., Veselov, A.P.:
Deformed Macdonald-Ruijsenaars operators and super Macdonald polynomials. 
{\em Commun.~Math.~Phys.}~288, 653--675 (2009).

\end{thebibliography}
\end{document}